\documentclass[10pt]{amsart}
\usepackage{amsfonts}
\usepackage{array}
\usepackage{tabularx}
\usepackage{arydshln}
\usepackage{amsmath}

\usepackage{amssymb}
\usepackage{physics}
\usepackage{graphicx}
\usepackage{caption}
\usepackage{subcaption}
\usepackage{multirow}
\usepackage{footmisc}
\usepackage{xcolor}

\providecommand{\U}[1]{\protect\rule{.1in}{.1in}}

\usepackage{xcolor}
\usepackage{latexsym}
\usepackage{amsmath}
\usepackage{amssymb}
\usepackage{mathrsfs}
\usepackage{graphicx}
\usepackage{color}
\usepackage{pgfpages}
\usepackage{ifthen}
\usepackage{leftidx,tensor}
\usepackage[T1]{fontenc}
\usepackage[latin1]{inputenc}
\usepackage{mathtools}
\usepackage{comment}
\usepackage{dsfont}
\usepackage[nocompress]{cite}
\usepackage[shortlabels]{enumitem}
\usepackage{aliascnt}
\usepackage[bookmarks=true,pdfstartview=FitH, pdfborder={0 0 0}, colorlinks=true,citecolor=red, linkcolor=blue]{hyperref}
\usepackage{nicefrac}
\newtheorem{thm}{Theorem}[section]
\newaliascnt{cor}{thm}
\newaliascnt{prop}{thm}
\newaliascnt{lem}{thm}
\newtheorem{cor}[cor]{Corollary}
\newtheorem{prop}[prop]{Proposition}
\newtheorem{lem}[lem]{Lemma}
\aliascntresetthe{cor}
\aliascntresetthe{prop}
\aliascntresetthe{lem}

\newaliascnt{defn}{thm}
\newaliascnt{asu}{thm}
\newaliascnt{con}{thm}

\newtheorem{asu}[asu]{Assumption}

\aliascntresetthe{defn}
\aliascntresetthe{asu}
\aliascntresetthe{con}

\newcounter{stp}
\newcounter{stpi}
\newcounter{stpci}
\newcounter{stpiii}

\newaliascnt{rem}{thm}
\newaliascnt{exa}{thm}
\newaliascnt{masu}{thm}
\newaliascnt{nota}{thm}
\newaliascnt{sett}{thm}
\newtheorem{rem}[rem]{Remark}
\newtheorem{exa}[exa]{Example}

\aliascntresetthe{rem}
\aliascntresetthe{exa}
\aliascntresetthe{masu}
\aliascntresetthe{nota}
\aliascntresetthe{sett}
\numberwithin{equation}{section}

\setlist[enumerate]{font = \normalfont}

\newcommand {\R}	{\mathbb{R}}

\newcommand {\T}	{\mathbb{T}}
\newcommand {\mbbP}	{\mathbb{P}}

\renewcommand{\d}{\, \mathrm{d}}

\renewcommand{\H}{\mathrm{H}}

	\renewcommand{\phi}{\varphi}
	
	\renewcommand{\div}{\mathrm{div} \, }

	\newcommand{\rC}{\mathrm{C}}
	\newcommand{\rL}{\mathrm{L}}
	
	\newcommand{\rH}{\H}

       \newcommand{\cV}{\mathcal{V}}

\title[Continuous Data Assimilation for 2D NSEs from Partial Tangential Boundary Observations]{Continuous Data Assimilation for the 2D Navier-Stokes Equations from Partial Tangential Boundary Observations}

\author{Gianmarco Del Sarto}
\address{Technische Universit\"{a}t Darmstadt\\
Fachbereich Mathematik\\
	Schlossgartenstr.\ 7\\
	64289 Darmstadt\\
	Germany}
\email{delsarto@mathematik.tu-darmstadt.de}

\author{Buddhika Priyasad}
\address{Universit\"{a}t Konstanz\\
Fachbereich Mathematik und Statistik\\
	Universit\"atstr.\ 10\\
	78457 Konstanz\\
	Germany}
\email{priyasad@uni-konstanz.de}

\begin{document}
\keywords{Continuous data assimilation; two-dimensional Navier-Stokes equations; Navier-slip boundary conditions; partial tangential boundary observations; feedback-induced coercivity}

\subjclass[2020]{Primary 35Q30; Secondary 93C20, 35B40, 76D05, 93D15}


\begin{abstract}
We study continuous data assimilation for the two-dimensional Navier--Stokes
equations on a smooth, bounded, connected domain with Navier-slip boundary
conditions, using no interior observations. The available data consist only
of finite-dimensional measurements of the tangential velocity on a non-empty
relatively open subset $\Gamma\subset\partial\Omega$. We prove that
sufficiently strong boundary feedback, constructed from sufficiently fine
observations, generates a coercive spectral gap for the assimilation error.
The limiting gap is identified with that of a mixed-boundary problem obtained
by imposing a homogeneous Dirichlet condition on $\Gamma$, and is shown to be
of order $\nu$. Combining this feedback-induced coercivity with an estimate of
the non-linear error production in terms of the long-time averaged
symmetric-gradient energy of the reference solution, we obtain a sufficient
criterion for exponential synchronisation. We verify this criterion in the
unforced case, for sufficiently small forcing when the unnudged Navier-slip
form has an $\rL^2$ spectral gap, for sufficiently large viscosity on domains
without tangential rigid motions, and for sufficiently large viscosity in the
presence of positive boundary friction. We also treat perfect slip on domains
admitting tangential rigid motions, where synchronisation follows under a
smallness condition on the non-rigid solenoidal component of the forcing.
\end{abstract}

\maketitle

\section{Introduction and main results}
\subsection{Boundary data assimilation for the Navier-Stokes equations}
\noindent 
Data assimilation aims to recover the state of an evolution equation from partial observations by coupling the underlying model dynamics with a feedback term determined by the discrepancy between the observed data and the corresponding model predictions \cite{Bouttier,Stuart2015}. This problem is particularly important in geophysical applications, where the available observations are typically sparse relative to the dimension and complexity of the underlying atmospheric or oceanic models \cite{Ghil,Kalnay2003AtmosphericDA,Carrassi}.

In this work, we study data assimilation for the two-dimensional Navier-Stokes equations (2D NSEs) with Navier-slip boundary conditions, assuming that observations are available \emph{only} on a subset of the boundary. Let $\Omega\subset\mathbb{R}^2$ be a smooth, bounded, and connected domain, and let $\Gamma\subset\partial\Omega$ be a non-empty relatively open subset corresponding to the observed part of the boundary. The reference velocity is governed by the system
\begin{equation}
\left\{
\begin{aligned}
\partial_t u-\nu\Delta u+(u\cdot\nabla)u+\nabla p&=g,
&&\text{in }\Omega,\\
\nabla\cdot u&=0,
&&\text{in }\Omega,\\
u\cdot n&=0,
&&\text{on }\partial\Omega,\\
\bigl(2\nu D(u)n\bigr)_\tau+\alpha u_\tau&=0,
&&\text{on }\partial\Omega,\\
u(0)&=u_0,
\end{aligned}
\right.
\label{eq: 2D nse}
\end{equation}
Here, $\nu >0$ denotes the viscosity, $p$ is the pressure, $g $ is a time-independent force, and $u_0$ is the initial condition. Moreover, $n$ denotes the outward unit normal to $\partial \Omega$ and $D(u) := ( \nabla u + \nabla u^T)/2$ is the symmetric gradient. Lastly, denoting by $\tau:= n^\perp$ the unit tangent to $\partial \Omega$, then $u_\tau := u \cdot \tau $ is the scalar tangential velocity component on the boundary, and $\alpha \geq 0$ is the boundary-friction coefficient. Well-posedness results for the 2D NSEs with Navier boundary conditions as in \eqref{eq: 2D nse} can be found in \cite{Kelliher2006}.

The boundary condition $u \cdot n = 0$ in \eqref{eq: 2D nse} is the impermeability condition: it prevents the fluid from crossing the boundary. Differently from the no-slip condition, it does not require the tangential velocity $u_\tau$ to vanish. The second boundary condition in \eqref{eq: 2D nse}$_4$ describes the tangential interaction between the fluid and the boundary. The
vector $2\nu D(u)n$ is the viscous stress, and its tangential component
measures the shear stress applied on the boundary. Thus, \eqref{eq: 2D nse}$_4$ means that the tangential stress opposes the motion of the fluid and is
proportional to its tangential velocity. When $\alpha=0$, the tangential stress vanishes and the fluid can slide freely
along the boundary. When $\alpha>0$, the boundary resists this motion and adds
the dissipation $\alpha\|u_\tau\|_{\rL^2(\partial\Omega)}^2$ to the energy
balance. Larger values of $\alpha$ therefore produce stronger tangential
friction.

In continuous time, data assimilation is often implemented through \emph{nudging}: an auxiliary state evolves according to the same governing equations as the reference solution and is continuously corrected by a feedback term based on the discrepancy between the measured quantities and their predictions \cite{AOT2014}. Under suitable conditions on the observations and the feedback strength, this mechanism drives the auxiliary state toward the reference trajectory. We adapt this principle to the setting in which the available measurements are confined to a portion of the boundary. The restriction to boundary observations is motivated by the fact that, in many fluid-mechanical applications (for instance oceanography), measurements within the interior of the flow are difficult or impossible to obtain, whereas data collected on accessible portions of the boundary are considerably more realistic from an experimental and operational point of view  \cite{Sonnewald_2021}.

More precisely, we assume that measurements are available on
$\Gamma\subset\partial\Omega$. Since the normal velocity is already prescribed
by the impermeability condition, we take the measured boundary quantity to be
the tangential velocity. We therefore introduce the trace
$
T_\Gamma u:=u_\tau|_\Gamma$
and consider the measurements 
$$
y_\delta(t):=P_\delta T_\Gamma u(t),
$$
where $P_\delta:\rL^2(\Gamma)\to Y_\delta $ is a finite-dimensional observation operator. The parameter $\delta>0$ represents the observation resolution, with smaller values of $\delta$ corresponding to finer measurements.
The assimilated velocity $v$ is defined as the solution of the boundary-feedback system 
\begin{equation}
\left\{
\begin{aligned}
\partial_t v-\nu\Delta v+(v\cdot\nabla)v+\nabla q&=g,
&&\text{in }\Omega,\\
\nabla\cdot v&=0,
&&\text{in }\Omega,\\
v\cdot n&=0,
&&\text{on }\partial\Omega,\\
\bigl(2\nu D(v)n\bigr)_\tau+\alpha v_\tau
+\mu P_\delta^*\bigl(P_\delta T_\Gamma v-y_\delta\bigr)&=0,
&&\text{on }\Gamma,\\
\bigl(2\nu D(v)n\bigr)_\tau+\alpha v_\tau&=0,
&&\text{on }\partial\Omega\setminus\Gamma,\\
v(0)&=v_0.
\end{aligned}
\right.
\label{eq: da_model}
\end{equation}
The boundary conditions in \eqref{eq: da_model}, and the corresponding error
boundary conditions below, are understood through the variational formulations;
no separate classical stress trace is asserted for weak solutions. Here, $q$ denotes the pressure, $v_0$ is an arbitrary initial condition, and $\mu >0$ is the feedback strength. On the unobserved part $\partial \Omega\setminus \Gamma$, the velocity $v$ satisfies the same Navier-slip condition as $u$. On the observed part $\Gamma$, this condition is modified by the feedback term 
$$
\mu P_\delta^*\bigl(P_\delta T_\Gamma v-y_\delta\bigr).$$
The quantity $P_\delta T_\Gamma v-y_\delta\in Y_\delta$ is the difference between the observations predicted by $v$ and the measured observations of $u$ given by $y_{\delta}$. The adjoint operator $P_\delta^*$ maps this difference back to a force acting on the boundary. Thus, the feedback continuously corrects the tangential velocity of $v$ on $\Gamma$, with larger values of $\mu$ producing a stronger correction.

Our aim is to investigate the data assimilation error $$w:=v-u.$$
The objective is to prove exponential synchronisation of the assimilated solution $v$ with the reference solution $u$ in the energy norm. More precisely, we look for sufficient conditions on the feedback strength $\mu$, the observation resolution $\delta $, and the reference solution $u$ under which there exist constants $C,\gamma>0$, independent of time, such that
$$
\|w(t)\|_{\rL^2(\Omega)}
\leq C e^{-\gamma t}\|v_0-u_0\|_{\rL^2(\Omega)} \qquad \text{for all }t\geq0.$$
This estimate shows that the assimilation error converges to zero exponentially and that $v$ recovers the reference trajectory for any initial condition $v_0$ of the data assimilated system.

\subsection{Related works and main contribution}

The classical background for continuous data assimilation is the
Foias-Prodi theory of determining modes and its later extensions to
determining nodes, and volume elements \cite{FoiasProdi1967,JonesTiti1992}. These results show that the long-time
dynamics of the 2D NSEs can be determined by
finitely many degrees of freedom. The continuous data-assimilation algorithm by Azouani, Olson, and Titi (AOT-algorithm)
turns this principle into a continuous-in-time reconstruction method by
inserting the discrepancy between observed and predicted coarse data as a
feedback term, see \cite{AOT2014}. In the standard setting, the observations are
interior modes, nodal values, or local averages, and the associated
interpolant approximates the state inside the fluid domain. The AOT algorithm has been analysed, possibly with extensions, for many PDEs models, such as the 3D NSEs \cite{Biswas}, the 3D Ladyzhenskaya model \cite{Cao2022}, the 3D primitive equations \cite{Pei2019}, reaction-diffusion equations \cite{Larios}, and abstract evolution equations \cite{CDA}. At the same time, a large literature also concerns the continuous data assimilation with noisy observations, see \cite{Bessaih2015,Brocker_DelSarto_26} and the references therein.

Boundary effects also arise in related, but structurally different, settings. Wu, Dong, and Wang \cite{WuDongWang2025} consider a regularised 2D Navier--Stokes problem on a bounded domain whose boundary is divided into a no-slip portion and a portion subject to a non-linear friction-type slip law. Their continuous data-assimilation scheme nevertheless uses a general interpolant $I_{\hat h}$ of the velocity in the fluid domain and introduces the discrepancy
$
\mu\bigl(I_{\hat h}u^\varepsilon-I_{\hat h}v^\varepsilon\bigr)
$
as a distributed forcing in the momentum equation. They establish well-posedness and exponential convergence in $\rL^2(\Omega)$ for sufficiently strong feedback and sufficiently fine observations, and also study the recovery of an unknown viscosity and several numerical variants. Thus, in that work the slip condition belongs to the physical model, whereas neither the observations nor the nudging term are supported on the boundary.

Boundary feedback stabilisation provides another related, but structurally distinct, line of work. The general boundary-feedback theory for parabolic equations goes back at least to \cite{Triggiani1980Boundary}. For the Navier--Stokes equations, tangential boundary feedback stabilisation near unstable equilibria was developed in 
\cite{BarbuLasieckaTriggiani2006Memoir,BarbuLasieckaTriggiani2006Abstract}. Finite-dimensional localised tangential boundary controls, in combination
with localised interior controls, were subsequently used for the Oseen
equation \cite{LasieckaTriggiani2015Oseen} and the nonlinear
Navier-Stokes equations \cite{LasieckaTriggiani2015NSE}. In
\cite{LasieckaPriyasadTriggiani2021}, uniform stabilisation of the 3D
Navier-Stokes equations near an unstable equilibrium was obtained in
low-regularity Besov spaces by means of finite-dimensional localised
tangential-like boundary feedback, coupled with a finite-dimensional localised interior control. For non-stationary target trajectories,
\cite{Rodrigues2015} constructs a finite-dimensional Dirichlet boundary
feedback supported on a prescribed open subset $\Gamma_{\rm c}\subset\partial\Omega$.

These results concern state-feedback stabilisation: the target equilibrium or trajectory is prescribed in advance, and the controller is determined from the state perturbation relative to that target. In particular, the boundary observability inequalities entering such constructions are control-theoretic design tools and do not represent a stream of partial boundary measurements. Boundary traces also occur in inverse-problem results, for instance in stability and reconstruction estimates based on measurements of the velocity
and the Cauchy stress \cite{Badra2016}, while available Navier-Stokes observer constructions use interior nodal or averaged observations rather than boundary traces \cite{Sergiy_2023}.

The present work addresses a different problem. The available data consist exclusively of the finite-dimensional tangential observations $P_\delta T_\Gamma u$ on an open subset $\Gamma \subset \partial\Omega$, and the adjoint operator $P_\delta^*$ lifts the observational residual back into the Navier-slip law on the same boundary subset. To the best of our knowledge, no previous rigorous result establishes exponential synchronisation for the 2D NSEs under this precise combination of boundary-only observations, finite-dimensional tangential measurements, localisation to a boundary subset, and feedback acting directly through the boundary condition. The perfect-slip case considered below also exhibits a regime in which the boundary feedback is essential: it damps tangential rigid motions that are invisible to the natural viscous dissipation.

\subsection{Main results}
\noindent 
Let $H$ and $V$ denote the usual divergence-free subspaces of
$\rL^2(\Omega;\mathbb{R}^2)$ and $\rH^1(\Omega;\mathbb{R}^2)$, respectively, incorporating
the impermeability condition $z\cdot n=0$ on $\partial\Omega$. For $z\in V$, we denote by
$
T_\Gamma z:=z_\tau|_{\Gamma}
$
its tangential trace on the observed boundary portion $\Gamma$.

The finite-dimensional boundary observations are assumed to become increasingly
resolved as $\delta\to0$ in the following sense.
\begin{asu}[Dense boundary observations]
\label{cond:dense_boundary_observations}
For every $\delta>0$, let $Y_\delta$ be a finite-dimensional Hilbert space and
let
$$
P_\delta : \rL^2(\Gamma)\to Y_\delta
$$
be bounded. We denote by
$
P_\delta^*:Y_\delta\to\rL^2(\Gamma)
$
the adjoint of $P_\delta$. We assume that there is a constant $C_P>0$, independent of
$\delta$, such that
$$
\| P_\delta f \|_{Y_\delta} \leq C_P \| f \|_{ \rL^2 ( \Gamma )}
\qquad \text{for all } f \in  \rL^2( \Gamma),
$$
and that for $\delta\to0$
\begin{equation*}
\| P_\delta f\|_{ Y_ \delta}^2
\to \|f\|_{\rL^2(\Gamma)}^2 \qquad \text{for every } f \in \rL^2(\Gamma).
\end{equation*}
\end{asu}
\noindent 
Let $w=v-u$ be the assimilation error. Subtracting the reference model \eqref{eq: 2D nse} and
assimilated system \eqref{eq: da_model} and testing by $w$, we obtain the error energy balance
\begin{equation}
    \frac12\frac{\d}{\d t}
\|w(t)\|_{\rL^2(\Omega)}^2 + a_{\mu,\delta}(w(t),w(t))
= -\int_\Omega (w\cdot\nabla)u\cdot w \d x,
\label{eq: energy balance intro}
\end{equation}
where
$$
a_{\mu,\delta}(z,z)
:=
2\nu\|D(z)\|_{\rL^2(\Omega)}^2
+
\alpha\|z_\tau\|_{\rL^2(\partial\Omega)}^2
+
\mu\|P_\delta T_\Gamma z\|_{Y_\delta}^2.
$$
Thus $a_{\mu,\delta}$ collects the three dissipative mechanisms acting on the
error: viscosity, Navier boundary friction, and boundary feedback.

The first
main step of our work is to show that sufficiently strong and sufficiently
resolved feedback makes this quadratic form coercive in the full
$\rL^2(\Omega)$ norm. The importance of this coercivity is seen directly from the error energy balance \eqref{eq: energy balance intro}. If, for some $\kappa>0$,
$$
a_{\mu,\delta}(z,z) \geq \kappa\|z\|_{\rL^2(\Omega)}^2 \qquad\text{for every }z\in V,
$$
then the linear part of the error equation produces damping at rate $\kappa$. On the other hand, it can be checked that the non-linear term can be bounded in terms of the symmetric gradient of the reference solution, leading to a differential inequality of the form
$$ \frac{\d}{\d t}\|w(t)\|_{\rL^2(\Omega)}^2 + \left(\kappa-
C_{\kappa,\nu}\|D(u(t))\|_{\rL^2(\Omega)}^2 \right) \|w(t)\|_{\rL^2(\Omega)}^2 \leq0. $$
Consequently, if the feedback-induced gap $\kappa$ dominates the long-time average of the non-linear contribution, an averaged Gronwall argument yields the exponential decay of $w$ and hence synchronisation.

To describe the maximal spectral gap that can be produced by observations on
$\Gamma$, set
$$
a_\alpha(z,z)
:=
2\nu\|D(z)\|_{\rL^2(\Omega)}^2
+
\alpha\|z_\tau\|_{\rL^2(\partial\Omega)}^2, \qquad 
\kappa_{\infty,\Gamma}
:=
\inf_{\substack{z\in V\setminus\{0\} \\ T_\Gamma z=0}}
\frac{a_\alpha(z,z)}
{\|z\|_{\rL^2(\Omega)}^2}.
$$
Since every $z\in V$ already satisfies $z\cdot n=0$ on $\partial\Omega$, the
constraint $T_\Gamma z=0$ forces the full trace of $z$ to vanish on $\Gamma$. The quantity $\kappa_{\infty,\Gamma}$ is therefore the spectral gap of a
mixed-boundary problem with a homogeneous Dirichlet condition on $\Gamma$ and
the original Navier-slip condition on
$\partial\Omega\setminus\Gamma$.

The following result, whose proof is divided between \autoref{lem:small_viscosity_kappa_scaling} and \autoref{prop:feedback_induced_coercivity}, makes this mechanism
quantitative: it shows that the limiting mixed-boundary problem has a strictly
positive spectral gap of order $\nu$, and that every smaller gap can be
transferred to the actual finite-dimensional feedback by choosing the feedback
strength sufficiently large and the observations sufficiently resolved.

\begin{thm}[Feedback-induced coercivity]
\label{thm:main_boundary_feedback_spectral_gaps}
Let $\Gamma\subset\partial\Omega$ be non-empty and relatively open, and assume
\autoref{cond:dense_boundary_observations}. Then there exist positive constants
$c_\Gamma = c_\Gamma (\Omega, \Gamma),\, C_\Gamma =  C_\Gamma (\Omega, \Gamma)$ such that
$$
c_\Gamma\nu
\leq
\kappa_{\infty,\Gamma}
\leq
C_\Gamma\nu
$$
for every $\nu>0$ and every $\alpha\geq0$. Moreover, for every
$
0<\kappa<\kappa_{\infty,\Gamma},
$ there exist $\mu_0>0$ and $\delta_0>0$ such that, for every $\mu\geq\mu_0$ and every
$0<\delta\leq\delta_0$,
$$
a_{\mu,\delta}(z,z)
\geq
\kappa\|z\|_{\rL^2(\Omega)}^2
\qquad
\text{for every }z\in V.
$$
\end{thm}

\noindent 
It remains to compare the feedback-induced damping with the non-linear term in
the error energy balance in \eqref{eq: energy balance intro}. Since pointwise,
$
(w\cdot\nabla)u\cdot w
=
w^T D(u) w,
$
then the non-linear interaction in \eqref{eq: energy balance intro} is governed by the symmetric gradient $D(u)$. We therefore introduce the
long-time averaged symmetric-gradient energy
\begin{equation}
M_u^D
:=
\limsup_{t\to\infty}
\frac1t
\int_0^t
\|D(u(s))\|_{\rL^2(\Omega)}^2\,\d s.
\label{eq:def_MuD}
\end{equation}
For $\kappa>0$ and $\nu>0$, set
\begin{equation}
\Theta_{\kappa,\nu}
:=
C_4^2 C_{\mathrm{Korn}}
\left(
\kappa^{-1}+(2\nu)^{-1}
\right),
\label{eq:def_theta_kappa_nu_intro}
\end{equation}
where $C_4$ and $C_{\mathrm{Korn}}$ are the constants in the Ladyzhenskaya and
Korn inequalities, see \eqref{eq:ladyzhenskaya} and
\eqref{eq:korn_second_squared} respectively.

We are now ready to present the main convergence result of our work. It states that the feedback-induced coercivity gives exponential synchronisation whenever the resulting damping dominates the long-time averaged non-linear contribution of the reference flow, i.e. when $\Theta_{\kappa,\nu}M_u^D<\kappa$ for some $0<\kappa<\kappa_{\infty,\Gamma}$.
\begin{thm}[Exponential synchronisation under averaged symmetric-gradient control]
\label{thm:synchronisation}
Assume \autoref{cond:dense_boundary_observations}. Let $u$ be the solution of
\eqref{eq: 2D nse} with $u_0 \in H$, and assume that $M_u^D<\infty$. Let
$$
y_{\delta}(t) = P_{\delta} T_\Gamma u(t).
$$
For any $v_0 \in H$, let $v$ be the solution of the assimilated system in
\eqref{eq: da_model}. Suppose that there exists $\kappa$ such that
\begin{equation}
0<\kappa<\kappa_{\infty,\Gamma}
\qquad \text{and}\qquad
\Theta_{\kappa,\nu}M_u^D<\kappa.
\label{eq:non-linear_gap_condition}
\end{equation}
Then there exist $\mu_0>0$ and $\delta_0>0$ such that the following holds. For every
$\mu \geq\mu_0$, every $0 < \delta \leq \delta_0$,  and every $ 0< \gamma < \frac{1}{2}( \kappa- \Theta_{\kappa, \nu} M_u^D), $ there exists a constant $C>0$ such that

$$
\|v(t)-u(t)\|_{\rL^2(\Omega)}
\leq
C e^{-\gamma t}\| v_0-u_0 \|_{ \rL^2(\Omega)}
\qquad \text{for all } t \geq 0.
$$
The constant $C$ is independent of $t,v_0$, and of the particular admissible values of $\mu$ and $\delta$.
\end{thm}
\noindent 
The abstract condition \eqref{eq:non-linear_gap_condition} is automatically satisfied in the absence of external
forcing, because it can be proved that $M_u^D = 0 $ when $ g= 0.$

\begin{cor}[The unforced case]
\label{cor:unforced_synchronisation}
Assume \autoref{cond:dense_boundary_observations} and let $g=0$. Then, for
every fixed viscosity $\nu>0$ and every
$\kappa\in(0,\kappa_{\infty,\Gamma})$, there exist $\mu_0>0$ and
$\delta_0>0$ such that the conclusion of \autoref{thm:synchronisation} holds
for every $\mu\geq\mu_0$ and every $0<\delta\leq\delta_0$.
\end{cor}
\noindent 
In addition to the unforced case, \autoref{thm:synchronisation} gives the following four concrete synchronisation criteria for forced flows.
\begin{enumerate}[(i)]
\item \emph{Large viscosity in the absence of tangential rigid motions.}
Suppose that the only rigid motion satisfying the impermeability condition on
$\partial\Omega$ is the zero field, and let $C_\Omega$ be the constant from the
Korn-Poincar\'e inequality, see \autoref{lem:korn_poincare_no_rigid_motions}. If
$$
\nu^4
>
\frac{C_4^2 C_\Omega^3}{8}
    \|g\|_{\rL^2(\Omega)}^2,
$$
then exponential synchronisation holds, see
\autoref{prop:large_viscosity_synchronisation}.
\item \emph{Large viscosity with positive boundary friction.}
Suppose that $\alpha>0$, fix $\underline\nu>0$, and let
$\lambda_{\alpha,\underline\nu}$ and $C_{\alpha,\underline\nu}$ be the
constants in \autoref{lem:positive_alpha_coercivity}. If
$$
\nu
>
\max\left\{
\frac{C_4^2 C_{\alpha,\underline\nu}}
{\lambda_{\alpha,\underline\nu}^2}
\|g\|_{\rL^2(\Omega)}^2,
\underline\nu
\right\},
$$
then exponential synchronisation holds. In particular, positive boundary
friction removes the rigid-rotation obstruction, so this criterion applies
also to disks and concentric annuli, see
\autoref{prop:large_viscosity_positive_alpha}.
\item \emph{Small forcing under an unnudged spectral gap.}
Suppose that there exists $\lambda_\alpha>0$ such that
$$
a_\alpha(z,z)
\geq
\lambda_\alpha\|z\|_{\rL^2(\Omega)}^2
\qquad
\text{for every }z\in V.
$$
If, for some $\kappa\in(0,\kappa_{\infty,\Gamma})$,
$$
\| g \|_{ \rL^2(\Omega)}^2
<
\frac{2\nu\lambda_\alpha\kappa}{\Theta_{\kappa,\nu}},
$$
then exponential synchronisation holds, see
\autoref{prop:small_forcing_synchronisation}.
\item \emph{Perfect slip with tangential rigid motions.}
Suppose that $\alpha=0$ and
$\mathcal R_\Omega:=\{r\in V:D(r)=0\}\neq\{0\}$. Let
$\mbbP_H$ denote the Helmholtz projection, let $\Pi_{\mathcal R}$ be the
$\rL^2$-orthogonal projection onto $\mathcal R_\Omega$, and set
\[
g_\perp=(I-\Pi_{\mathcal R})\mbbP_Hg.
\]
Let $C_{\mathcal{R}}$ be the constant in the Korn-Poincar\'e inequality modulo $\mathcal{R}_\Omega$, see \autoref{lem:korn_poincare_modulo_rigid_motions}. If, for some $\kappa\in(0,\kappa_{\infty,\Gamma})$,
\[
\frac{C_{\mathcal R}\Theta_{\kappa,\nu}}{4\nu^2}
\|g_\perp\|_{\rL^2(\Omega)}^2<\kappa,
\]
then exponential synchronisation holds, see
\autoref{prop:perfect_slip_rigid_motions}.
\end{enumerate}

\subsection{Organisation of the paper}
This paper is organised as follows. In \autoref{sec: boundary-feedback}, we introduce the functional setting, formulate the reference and assimilated Navier-Stokes systems, and derive the evolution equation and energy balance for the assimilation error. In \autoref{sec: spectral gaps}, we study the spectral gaps generated by the boundary feedback and establish the resulting coercivity for sufficiently strong and sufficiently resolved observations. In \autoref{sec:synchronisation}, we combine this coercivity with the error energy estimate to prove the abstract exponential synchronisation theorem and derive the concrete large-viscosity, small-forcing, and perfect-slip regimes. Finally, \autoref{sec: appendix} collects the auxiliary results related to Korn and trace inequalities, well-posedness, and the averaged Gronwall argument.

\section{Boundary-feedback formulation and error equation}
\label{sec: boundary-feedback}
\subsection{Functional setting and boundary observations}
We start by fixing the spaces and notation used throughout the rest of the
paper. Let
$$
\cV := \left\{ \phi\in \rC^\infty (\overline {\Omega}; \R^2): \nabla \cdot \phi=0 \text{ in }\Omega,\, \phi \cdot n = 0 \text{ on }\partial\Omega \right\},
$$
and set
$$
H:=\overline{\cV}^{\, \rL^2(\Omega; \R^2)}, \qquad V:=\overline{\cV}^{\, \rH^1(\Omega; \R^2)}.
$$
For $z\in V$, we denote by
$$
T_\Gamma z:=z_\tau|_\Gamma
$$
the tangential trace on the observed part of the boundary. We also write
$$
b(a,b,c):=\int_\Omega (a \cdot \nabla) b \cdot c \d x,
$$
and define
\begin{equation}
    a_\alpha(z,\phi) := 2\nu\int_\Omega D(z):D(\phi) \d x +\alpha \int_{\partial\Omega} z_\tau \phi_\tau \d \sigma ,
    \label{eq: definition of a alpha}
\end{equation}
while, for the assimilated and error systems,
$$
a_{\mu,\delta}(z,\phi) := a_\alpha(z,\phi) + \mu \langle P_\delta T_\Gamma z, P_\delta T_\Gamma\phi\rangle_{Y_\delta}.
$$
The duality between $V'$ and $V$ is denoted by $\langle\cdot,\cdot\rangle_{V',V}$.

We next show two concrete families of finite-dimensional maps
$P_\delta \colon\rL^2(\Gamma) \to Y_\delta$ satisfying
\autoref{cond:dense_boundary_observations}.

\begin{exa}[Increasing orthogonal projections]
Let $(e_k)_{k \geq 1}$ be an orthonormal basis of $\rL^2(\Gamma)$, and let
$N_\delta \to \infty$ as $\delta \to0$. Set $Y_\delta:=\mathbb{R}^{N_\delta} $. Define
$$
P_\delta f:=
( (f,e_1)_{\rL^2(\Gamma)}, \dots, (f,e_{N_\delta})_{\rL^2(\Gamma)}).
$$
By Bessel's inequality,
$$
\|P_\delta f\|_{Y_\delta}^2 =
\sum_{k=1}^{N_\delta}
|(f,e_k)_{\rL^2(\Gamma)}|^2
\leq
\|f\|_{\rL^2(\Gamma)}^2.
$$
Moreover, by Parseval identity we have
$$
\|P_\delta f\|_{Y_\delta}^2
\to  \sum_{k=1}^{\infty}
|( f, e_k )_{\rL^2(\Gamma)}|^2
= \|f\|_{\rL^2(\Gamma)}^2.
$$
This example describes observations that measure a finite number of boundary
modes. For instance, if $(e_k)_{k\geq1}$ is a Fourier or another spectral
basis, then $P_\delta f$ records the first $N_\delta$ coefficients of the
tangential velocity. Finite collections of modes are standard interpolant observables in
continuous data assimilation \cite{AOT2014}. Moreover, Fourier representations
of wall data are used in wall-based flow reconstruction to identify the
large-scale components that remain observable away from the boundary
\cite{Encinar_2019}.
\end{exa}

\begin{exa}[Averages over boundary cells]
For every $\delta>0$, let
$(\Gamma_{\delta,j})_{j=1}^{N_\delta}$ be a measurable partition of
$\Gamma$ such that
$|\Gamma_{\delta,j}|>0$ and
$$
\max_{1\leq j \leq N_\delta} \operatorname{diam} ( \Gamma_{ \delta , j } )
\to 0 \qquad \text{as } \delta \to 0.
$$
Set $Y_\delta : = \R^{N_\delta}$ with its Euclidean norm and define
$$
(P_\delta f)_j := |\Gamma_{\delta,j}|^{-1/2}
\int_{\Gamma_{ \delta,j}} f\d \sigma.
$$
Then
$$
\|P_\delta  f \|_{Y_\delta}^2
= \sum_{j=1}^{N_\delta}
|\Gamma_{\delta,j}|
\vert  \frac1{|\Gamma_{\delta,j}|}
\int_{\Gamma_{\delta,j}}f\d \sigma
\vert  ^2.
$$
If
$$
A_\delta f :=
\sum_{j=1}^{N_\delta}
\left( \frac1{|\Gamma_{\delta,j}|} \int_{\Gamma_{ \delta, j}} f \d \sigma \right)
\mathds{1}_{ \Gamma_{\delta , j }},
$$
then
$$
\| P_\delta f \|_{Y_\delta} = \| A_\delta f\|_{\rL^2(\Gamma)}.
$$
Jensen's inequality yields
$$
\|A_\delta  f \|_{ \rL^2(\Gamma)}
 \leq \| f \|_{ \rL^2(\Gamma)}.
$$
Moreover, the vanishing mesh size implies that
$A_\delta g\to g$ in $\rL^2(\Gamma)$ for every continuous function $g$ on
$\overline{\Gamma}$. By density and the preceding contraction estimate, the
same convergence holds for every $f \in \rL^2(\Gamma)$. Consequently,
$$
\|P_\delta f\|_{Y_\delta}^2 = \|A_\delta f\|_{\rL^2(\Gamma)}^2
\to \|f\|_{\rL^2(\Gamma)}^2.
$$
This example represents local boundary sensors: each component of $P_\delta f$ is a normalised average of the tangential velocity over one boundary cell. Such observables are the boundary analogue of the local spatial averages used in continuous data assimilation \cite{AOT2014}. Differently from point evaluations, they are bounded on $\rL^2(\Gamma)$ and reflect measurements with finite spatial resolution.
\end{exa}
\noindent 
Both constructions in the previous two examples satisfy
\autoref{cond:dense_boundary_observations} with $C_P=1$.

\subsection{Reference and assimilated Navier-Stokes systems}
\label{subsec: well posedness reference and DA system}
With the functional setting fixed above, we recall the weak well-posedness result for
the reference 2D NSEs with Navier slip boundary conditions \eqref{eq: 2D nse}. Its proof, which is based on \cite[Theorem 6.1]{Kelliher2006}, is postponed to \autoref{appendix: well posedness}.
\begin{lem}[Well-posedness of the reference Navier-Stokes system]
\label{lem: well_posedness_reference}
Let $T>0$, $u_0\in H$, and let $g\in \rL^2(\Omega;\R^2)$ be time independent. Then the Navier-Stokes system
\eqref{eq: 2D nse} with Navier slip boundary condition admits a unique weak
solution $u$ such that
$$
u \in \rL^\infty(0,T;H) \cap \rL^2(0,T;V) \cap \rC([0,T];H), \qquad  u'\in \rL^{4/3}(0,T;V'). 
$$
It is characterised by $u(0)=u_0$ in $H$, and for every
$\phi\in V$
$$
\langle  u',\phi\rangle_{V',V} + a_\alpha(u,\phi) +
b(u,u,\phi) = (g,\phi)_{\rL^2(\Omega)}
$$
in $\mathcal{D}'(0,T)$. Moreover, $u$ satisfies the energy inequality
\begin{equation}
    \frac{1}{2} \| u(t) \|_{\rL^2(\Omega)}^2
+ \int_0^t a_\alpha(u(s),u(s))\d s \leq
\frac{1}{2} \| u_0 \|_{\rL^2(\Omega)}^2
+ \int_0^t (g,u(s))_{\rL^2(\Omega)}\d s
\label{eq: energy inequality 2DNSE}
\end{equation}
for every $t\in[0,T]$.
\end{lem}
\noindent 
We next show that adding the boundary feedback term preserves the standard
well-posedness theory. The detailed proof of this fact can be found in \autoref{appendix: well posedness}.
\begin{lem}[Well-posedness of the boundary data assimilation system]
\label{lem:wp_da_system}
Assume \autoref{cond:dense_boundary_observations}. Let $T>0$, $v_0\in H$, $g\in \rL^2(\Omega; \R^2)$, $\mu>0$, and $\delta>0$ be fixed. Assume moreover
$$
y_\delta\in \rL^2(0,T;Y_\delta).
$$
Then the data assimilation system \eqref{eq: da_model} admits a unique weak
solution $v$ satisfying
$$
v \in \rL^\infty(0,T;H)\cap \rL^2(0,T;V) \cap \rC([0,T];H),
\qquad  v' \in \rL^{4/3}(0,T;V').
$$
It is characterised by $v(0)=v_0$ in $H$, and for every
$\phi\in V$
$$
\langle v', \phi\rangle_{V',V} +a_{\mu,\delta}(v,\phi) + b(v,v,\phi)= (g,\phi)_{\rL^2(\Omega)}+ \mu \langle y_\delta, P_\delta T_\Gamma \phi \rangle_{Y_\delta}
$$
in $\mathcal{D}'(0,T)$. In particular, the boundary operator $P_\delta^*P_\delta T_\Gamma$ appearing in \eqref{eq: da_model} is understood through the
bounded, symmetric, non-negative form
$$
(z,\phi) \mapsto \mu \langle P_\delta T_\Gamma z, P_\delta T_\Gamma\phi\rangle_{Y_\delta}
$$
on $V\times V$.
\end{lem}

\subsection{Assimilation error and energy balance}
The data assimilation error $w:= v-u$ solves 
\begin{equation}
\left \lbrace 
\begin{aligned}
\partial_t w-\nu \Delta w+(v\cdot\nabla)w+(w\cdot\nabla)u+\nabla\pi&=0,
\qquad & \text{in } &\Omega,\\
\nabla\cdot w&=0,
& \text{in } &\Omega,\\
w\cdot n &=0,
& \text{on } &\partial\Omega,\\
\big(2\nu D(w)n\big)_\tau+\alpha w_\tau
+\mu P_\delta^*P_\delta T_\Gamma w &=0,
& \text{on } &\Gamma, \\
\bigl(2\nu D(w)n\bigr)_\tau+\alpha w_\tau&=0,
&\text{on } &\partial\Omega\setminus\Gamma,\\
w(0) &= v_0 - u_0,
\end{aligned}
\right.
\label{eq: error_equation}
\end{equation}
where $\pi$ denotes the pressure. The next lemma shows that the difference between the assimilated and reference solutions is the unique weak solution of the error equation in the natural energy class and satisfies the fundamental energy balance used in the synchronisation analysis.

\begin{lem}[Well-posedness of the error equation and energy balance]
\label{thm:wp_error_equation}
Let the assumptions of \autoref{lem: well_posedness_reference} and \autoref{lem:wp_da_system} hold, and suppose that the observations are generated by
the reference solution, namely
$$
y_\delta(t)=P_\delta T_\Gamma u(t) \qquad\text{for a.e. }t\in(0,T).
$$
Set $w_0:=v_0-u_0$.  For the fixed functions $u$ and $v$ given by
\autoref{lem: well_posedness_reference} and \autoref{lem:wp_da_system},
respectively, the linear error equation \eqref{eq: error_equation}
admits a unique weak solution $w$ satisfying
$$
w\in \rL^\infty(0,T;H)\cap \rL^2(0,T;V) \cap \rC([0,T];H),
\qquad w'\in \rL^{2}(0,T;V').
$$
It is characterised by $w(0)=w_0$ in $H$, and for every
$\phi\in V$
$$
\langle w', \phi \rangle_{ V', V} + a_{\mu,\delta}(w,\phi) + b(v , w, \phi) + b(w,u,\phi) =0
$$
in $\mathcal{D}'(0,T)$. Moreover, $w=v-u$, where $u$ and $v$ are the unique
solutions given by \autoref{lem: well_posedness_reference} and
\autoref{lem:wp_da_system}, and the following energy identity holds in the sense of
distributions on $(0,T)$
\begin{equation}
\frac{1}{2}\frac{\d}{\d t} \| w (t) \|_{\rL^2(\Omega)}^2
+ a_{ \mu, \delta}( w(t) , w(t))= - \int_\Omega (w(t) \cdot \nabla) u(t) \cdot w(t) \d x .
    \label{eq: error assimilation equality distribution}
\end{equation}
\end{lem}
\begin{proof}
Let $u$ and $v$ be the unique weak solutions given by
\autoref{lem: well_posedness_reference} and \autoref{lem:wp_da_system}, and set
$w:=v-u$. Then
$$
w \in \rL^\infty(0,T;H) \cap \rL^2(0,T;V) \cap \rC([0,T];H),
\qquad w'\in \rL^{4/3}(0,T;V'),
$$
and subtracting the two variational formulations gives, for every $\phi\in V$
\begin{equation}
    \langle w', \phi \rangle_{V',V} +a_{ \mu , \delta}(w,\phi) +b( v, w , \phi)+b( w , u ,\phi) =0
    \label{eq: weak_formulation error}
\end{equation}
in $\mathcal{D}'(0,T)$. Thus $w$ is a weak solution of \eqref{eq: error_equation}.

We first improve the time regularity of $w$. For $a,b\in V$, define
$B(a,b)\in V'$ by
$$
\langle B(a,b), \phi\rangle_{V',V} := b( a, b , \phi), \qquad   \phi \in V.
$$
Since $\nabla \cdot a=0$ in $\Omega$ and $a\cdot n=0$ on
$\partial \Omega$, integration by parts gives
$$
b(a, b ,\phi) =-b( a, \phi,b).
$$
Therefore, by Ladyzhenskaya's inequality,
$$ 
|b(a,b,\phi)| = |b(a,\phi,b)|\leq \| a \|_{ \rL^4(\Omega)} \| b \|_{\rL^4(\Omega)}
\| \nabla \phi\|_{ \rL^2(\Omega) }\leq C \| a \|_H^{1/2} \| a \|_V^{1/2}
\|b\|_H^{1/2} \| b\|_V^{1/2} \|\phi\|_V.
$$
Consequently, if$a,b\in\rL^\infty(0,T;H)\cap\rL^2(0,T;V),$
then $
B(a,b)\in\rL^2(0,T;V').$
Indeed,
\begin{equation*}
    \begin{split}
        \int_0^T\|B(a,b)\|_{V'}^2 \d t &\leq
C \| a \|_{\rL^\infty(0,T;H)} \| b \|_{\rL^\infty(0,T;H)} \int_0^T \| a(t) \|_V\| b(t) \|_V \d t\\ &\leq C \| a \|_{ \rL^\infty(0,T ; H)} \| b \|_{ \rL^\infty(0,T;H)} \| a \|_{\rL^2 (0,T ; V)}
\| b \|_{\rL^2(0,T;V)}.
    \end{split}
\end{equation*}
In particular
$ B(v,w),\ B(w,u)\in\rL^2(0,T;V'). $ Moreover, the boundedness of $a_{\mu,\delta}$ on $V\times V$ gives $ a_{\mu,\delta} (w,\cdot)\in\rL^2(0,T;V') .
$ It follows from \eqref{eq: weak_formulation error} that
$
w'\in\rL^2(0,T;V').
$

Since
$
w\in\rL^2(0,T;V), \ w'\in\rL^2(0,T;V'),
$
the standard chain rule for the Gelfand triple
$V\hookrightarrow H\hookrightarrow V'$ implies that
$t\mapsto\|w(t)\|_H^2$ is absolutely continuous and
$$
\frac{1}{2}\frac{\d}{\d t}\|w(t)\|_H^2 = \langle w'(t) , w(t)\rangle_{V',V}
$$
for almost every $t\in(0,T)$. Pairing
\eqref{eq: weak_formulation error} with $w(t)$ therefore gives
$$
\frac{1}{2} \frac{\d}{\d t} \|w(t)\|_H^2
+ a_{ \mu, \delta}( w(t), w(t)) + b( v(t) , w(t) , w(t)) + b(w(t) , u(t) , w(t)) = 0.
$$
Since $b(v,w,w) = 0$, we obtain
$$
\frac{1}{2}\frac{\d}{\d t} \| w(t) \|_{\rL^2(\Omega)}^2 + a_{ \mu , \delta}(w(t),w(t)) = -\int_\Omega
(w(t) \cdot \nabla ) u ( t) \cdot w(t) \d x,
$$
which proves
\eqref{eq: error assimilation equality distribution}.

It remains to prove uniqueness. Let $w_1,w_2$ be two weak solutions with the
same initial condition and set $e:=w_1-w_2$. The preceding regularity argument
gives
$
e\in\rL^2(0, T ; V), \ e' \in \rL^2( 0 , T; V').
$
The chain rule therefore permits us to test the difference equation by $e$,
and we obtain
$$
\frac{1}{2} \frac{\d}{\d t} \| e(t) \|_H^2 + a_{\mu, \delta }( e(t ) , e(t)) = -b( e( t ) ,u (t ) ,e (t)).
$$
By Ladyzhenskaya's and Young's inequalities,
$$
| b( e ,u , e)| \leq C \| u \|_V \| e \|_H \| e \|_V \leq \frac{c_0}{2}\|e\|_V^2
+ C \|u \|_V^2 \|e\|_H^2.
$$
Moreover, from Korn's inequality in \autoref{lem:korn_compact_trace} and the non-negativity of the feedback term we have
that there exist constants $c_0,C_0>0$ such that
$$
a_{\mu,\delta}( e, e)+C_0\|e\|_H^2 \geq c_0 \| e \|_V^2.
$$
Consequently,
$$
\frac{\d}{\d t} \| e(t) \|_H^2 \leq
C(1 + \| u(t) \|_V^2) \| e(t)\|_H^2.
$$
Since $u\in\rL^2(0,T;V)$ and $e(0)=0$, Gronwall's lemma yields $e=0$.
Hence the weak solution is unique and necessarily coincides with $w=v-u$.

\end{proof}

\section{Spectral gaps generated by boundary feedback}
\label{sec: spectral gaps}

\subsection{Three boundary-feedback spectral gaps}
\label{sec: feedback boundary coercivity}

This subsection introduces the coercive estimate that will be used in the error
energy identity \eqref{eq: error assimilation equality distribution}. In the
synchronisation argument in \autoref{sec:synchronisation}, the unknown is the
assimilation error $w(t)=v(t)-u(t)$. Here we write $z\in V$ for a generic
possible value of this error and consider the feedback-modified quadratic form 
$$
a_{\mu,\delta}(z,z)=a_\alpha(z,z)
+\mu\|P_\delta T_\Gamma z\|_{Y_\delta}^2 .
$$
The first term is the usual Navier-slip dissipation, while the second term
penalises the observed tangential trace on $\Gamma$. The parameter $\mu$
measures the feedback strength, whereas $\delta$ represents the observation
resolution.

The goal is to prove a lower bound
$$
a_{\mu,\delta}(z,z)\geq \kappa\|z\|_{\rL^2(\Omega)}^2,
\qquad z\in V,
$$
for some $\kappa>0$. Such a bound means that the feedback dissipation controls
the full $\rL^2$ size of the error. This is the key input in the
synchronisation argument: testing the error equation by $w$ gives formally
$$
\frac{1}{2}\frac{\d}{\d t}\|w(t)\|_{\rL^2(\Omega)}^2
+a_{\mu,\delta}(w(t),w(t))
\leq \text{non-linear terms},
$$
so coercivity turns the feedback contribution into linear damping.

We introduce three spectral gaps, corresponding respectively to the actual
finite-dimensional feedback, the idealised full-trace feedback, and the
limiting mixed-boundary constraint. The Rayleigh quotient associated with the
actual feedback-modified form is
$$
\kappa_{ \mu, \delta} := \inf_{ z \in V \setminus \{ 0 \}}
\frac{ a_{ \mu, \delta }( z , z )}{ \| z \|_{ \rL^2( \Omega) }^2 } .
$$
This is the spectral gap available to the data assimilation system at feedback
strength $\mu$ and observation resolution $\delta$. Proving a positive lower
bound for $\kappa_{\mu,\delta}$ is exactly what is needed to obtain linear
damping in the error equation.

To analyse this gap, for fixed $\mu>0$ we introduce the full-trace
boundary-penalised form
$$
a_{\mu,\Gamma}(z,z) := a_\alpha(z,z)+\mu\|T_\Gamma z\|_{ \rL^2 ( \Gamma)}^2.
$$
This is the idealised version of the feedback form in which the full
tangential trace on $\Gamma$ is available instead of the finite-dimensional
observation $P_\delta T_\Gamma z$. Its first Rayleigh quotient is
$$
\kappa_{\mu,\Gamma} :=
\inf_{z\in V\setminus \{ 0 \}} \frac{ a_{ \mu, \Gamma}(z,z)}{\| z \|_{ \rL^2( \Omega)}^2}.
$$
Equivalently, $\kappa_{\mu,\Gamma}$ is the largest number such that
$$
a_{\mu,\Gamma}(z,z)
\geq
\kappa_{\mu,\Gamma}\|z\|_{\rL^2(\Omega)}^2
\qquad
\text{for all }z\in V.
$$

Finally, the limiting mixed-boundary gap is
$$
\kappa_{ \infty, \Gamma} := \inf_{ \substack { z \in V \setminus \{ 0 \}\\ T_\Gamma z=0} }
\frac{a_\alpha(z,z)}{ \| z \|_{ \rL^2( \Omega) }^2}.
$$
The constraint $T_\Gamma z=0$ is the limiting effect of the penalty
$\mu\|T_\Gamma z\|_{\rL^2(\Gamma)}^2$ as $\mu\to\infty$: modes with non-zero
trace on $\Gamma$ become increasingly expensive, so bounded-energy sequences
are forced towards fields with vanishing tangential trace on $\Gamma$.
Accordingly, $\kappa_{\infty,\Gamma}$ is the limiting spectral benchmark for
the full-trace boundary-feedback mechanism.

Since every $z\in V$ already satisfies $z\cdot n=0$ on $\partial\Omega$, the
additional constraint $T_\Gamma z=0$ forces the full trace of $z$ to vanish on
$\Gamma$. At the level of the associated variational problem, this corresponds
to a mixed-boundary configuration with a homogeneous Dirichlet condition on
$\Gamma$ and the original Navier-slip condition on
$\partial\Omega\setminus\Gamma$.

The first result below establishes that the limiting mixed-boundary gap $\kappa_{\infty, \Gamma}$ is
strictly positive, and consists in the first part of the statement of \autoref{thm:main_boundary_feedback_spectral_gaps}. We then show, in \autoref{subsec: penalisation and observation limits}, that the full-trace gap
$\kappa_{\mu,\Gamma}$ increases to $\kappa_{\infty,\Gamma}$ as
$\mu\to\infty$, and that, for each fixed $\mu>0$, the discrete gap
$\kappa_{\mu,\delta}$ converges to $\kappa_{\mu,\Gamma}$ as the observation
resolution tends to zero. Combining these two approximation results yields the
feedback-induced coercivity estimate stated in
\autoref{prop:feedback_induced_coercivity}.
\begin{lem}[Viscosity scaling of the limiting gap]
\label{lem:small_viscosity_kappa_scaling}
There exist positive constants
$c_\Gamma = c_\Gamma (\Omega, \Gamma),\, C_\Gamma =  C_\Gamma (\Omega, \Gamma)$ such that, for every $\nu>0$ and every $\alpha\geq0$,
\begin{equation}
c_\Gamma\nu
\leq
\kappa_{\infty,\Gamma}
\leq
C_\Gamma\nu.
\label{eq:small_viscosity_kappa_scaling}
\end{equation}
\end{lem}
\begin{proof}
Set
$$
V_\Gamma:=\{z\in V:T_\Gamma z=0\}.
$$
Since every $z\in V_\Gamma$ has zero full trace on the non-empty relatively
open boundary part $\Gamma$, the Korn-Poincar\'e inequality with partial
Dirichlet boundary condition from \autoref{lem:korn_poincare_partial_dirichlet} gives a constant $C_{\mathrm{KP},\Gamma}>0$ such
that
$$
\|z\|_{\rL^2(\Omega)}^2 \leq C_{\mathrm{KP},\Gamma}\|D(z)\|_{\rL^2(\Omega)}^2
\qquad\text{for all }z\in V_\Gamma.
$$
Therefore
$$
a_\alpha(z,z)
\geq
2\nu\|D(z)\|_{\rL^2(\Omega)}^2
\geq
\frac{2\nu}{C_{\mathrm{KP},\Gamma}}
\|z\|_{\rL^2(\Omega)}^2,
$$
which proves the lower bound in \eqref{eq:small_viscosity_kappa_scaling}.

For the upper bound, choose a non-zero $\psi\in \rC_c^\infty(\Omega)$ and set
$z_0=\nabla^\perp\psi$. Then $z_0\in V$, its trace vanishes on the whole
boundary, and hence $T_\Gamma z_0=0$ and $(z_0)_\tau=0$ on $\partial\Omega$.
Thus
$$
\kappa_{\infty,\Gamma}
\leq
\frac{a_\alpha(z_0,z_0)}{\|z_0\|_{\rL^2(\Omega)}^2}
=
2\nu
\frac{\|D(z_0)\|_{\rL^2(\Omega)}^2}
{\|z_0\|_{\rL^2(\Omega)}^2},
$$
which is the required upper bound.
\end{proof}

\subsection{Penalisation and observation limits}
\label{subsec: penalisation and observation limits}
We now relate the three spectral gaps. First, we let $\mu \to \infty$ in the full-trace problem. Then, for fixed $\mu > 0$, we let $\delta \to 0$ to recover the full-trace gap from the discrete observations.

\begin{lem}[Infinite-feedback limit of the full-trace gap]
\label{lem:continuous_boundary_penalization_limit}
The map $\mu\mapsto\kappa_{\mu,\Gamma}$ is nondecreasing and
$$
\kappa_{ \mu,\Gamma} \uparrow \kappa_{\infty, \Gamma} \qquad \text{as } \mu \to \infty .
$$
\end{lem}

\begin{proof}
The monotonicity follows from the definition of $a_{\mu,\Gamma}$. Moreover
$$
\kappa_{\mu,\Gamma}\leq \kappa_{\infty,\Gamma}
$$
for every $\mu>0$, since the penalisation term vanishes on every admissible
function satisfying $T_\Gamma z=0$.

Let $\mu_j\to\infty$ and choose $z_j\in V$ such that
$$
\|z_j\|_{\rL^2(\Omega)}=1,
\qquad
a_{\mu_j,\Gamma}(z_j,z_j)
\leq
\kappa_{\mu_j,\Gamma}+\frac1j
\leq
\kappa_{\infty,\Gamma}+1 .
$$
By \autoref{lem:korn_compact_trace}, $(z_j)$ is bounded in $V$. After extracting a
subsequence,
$$
z_j\rightharpoonup z \quad\text{weakly in }V,
\qquad
z_j\to z \quad\text{strongly in }\rL^2(\Omega),
$$
and therefore $\|z\|_{\rL^2(\Omega)}=1$. Also,
$$
\mu_j\|T_\Gamma z_j\|_{\rL^2(\Gamma)}^2
\leq
\kappa_{\infty,\Gamma}+1,
$$
so $T_\Gamma z_j\to0$ in $\rL^2(\Gamma)$. By the compactness of $T_\Gamma$,
$T_\Gamma z_j\to T_\Gamma z$ strongly in $\rL^2(\Gamma)$ along the same
subsequence. Hence $T_\Gamma z=0$.

Using the lower semicontinuity of $a_\alpha$ and the choice of $z_j$, we obtain
$$
\liminf_{j\to\infty}\kappa_{\mu_j,\Gamma}
\geq
\liminf_{j\to\infty}a_{\mu_j,\Gamma}(z_j,z_j)
\geq
a_\alpha(z,z)
\geq
\kappa_{\infty,\Gamma}\|z\|_{\rL^2(\Omega)}^2
=
\kappa_{\infty,\Gamma}.
$$
Together with $\kappa_{\mu,\Gamma}\leq\kappa_{\infty,\Gamma}$, this concludes the proof.
\end{proof}
\noindent 
We next pass from the full-trace boundary penalisation to the actual discrete
observation term. The following consequence of \autoref{cond:dense_boundary_observations} will be used in the compactness
argument: if $\delta_j\to0$ and $f_j\to f$ in $\rL^2(\Gamma)$, then
\begin{equation}
\label{eq:observation norm_convergence}
\|P_{\delta_j}f_j\|_{Y_{\delta_j}}
\to
\|f\|_{\rL^2(\Gamma)}.
\end{equation}
Indeed,
$$
\left|
\|P_{\delta_j}f_j\|_{Y_{\delta_j}}-
\|f\|_{\rL^2(\Gamma)}
\right|
\leq
C_P\|f_j-f\|_{\rL^2(\Gamma)}+
\left|
\|P_{\delta_j}f\|_{Y_{\delta_j}}-
\|f\|_{\rL^2(\Gamma)}
\right|,
$$
and both terms on the right-hand side vanish as $j\to\infty$.
\begin{lem}[Dense-observation limit of the discrete feedback gap]
\label{lem:discrete_gap_convergence}
Assume \autoref{cond:dense_boundary_observations}. Fix $\mu>0$. Then
$$
\lim_{\delta \to 0^+} \kappa_{ \mu, \delta} = \kappa_{ \mu ,\Gamma}.
$$
\end{lem}
\begin{proof}
\emph{Step $1$: $\liminf_{\delta\to0^+}\kappa_{\mu,\delta}
\geq
\kappa_{\mu,\Gamma}$.} Let $\delta_j\to0$ be such that $
\kappa_{\mu,\delta_j}
\to
\liminf_{\delta\to0^+}\kappa_{\mu,\delta}.$ For each $j$, choose $z_j\in V$ with
$$
\| z_j \|_{ \rL^2(\Omega)}=1, \qquad a_{\mu, \delta_j}( z_j, z_j) \leq \kappa_{ \mu, \delta_j}+ \frac{1}{j} .
$$
The sequence $(\kappa_{\mu,\delta_j})$ is bounded from above. Indeed,
for any fixed $z_0\in V\setminus\{0\}$, the uniform boundedness of
$P_\delta$ in \autoref{cond:dense_boundary_observations} gives
$$
\kappa_{ \mu, \delta_j} \leq \frac{ a_{ \mu, \delta_j}( z_0 , z_0)}{ \| z_0 \|_{ \rL^2 ( \Omega ) }^2}
\leq \frac{
a_\alpha (z_0 , z_0 ) + \mu C_P^2 \| T_\Gamma z_0 \|_{ \rL^2(\Gamma ) }^2
}{ \| z_0 \|_{ \rL^2 ( \Omega) }^2
}.
$$
Hence $(a_{\mu,\delta_j}(z_j,z_j))$ is bounded. Since
$$
a_\alpha(z_j,z_j)\leq a_{\mu,\delta_j}(z_j,z_j),
$$
and $\|z_j\|_{\rL^2(\Omega)}=1$, \autoref{lem:korn_compact_trace} implies that
$(z_j)$ is bounded in $V$. Therefore, thanks to the reflexivity of the Hilbert space $V$ and the compact embedding $V \hookrightarrow \rL^2(\Omega),$ up to a subsequence we have
$$
z_j\rightharpoonup z  \quad  \text{weakly in }V, \qquad z_j \to z
\quad\text{strongly in }\rL^2(\Omega).
$$
In particular, $\|z\|_{\rL^2(\Omega)}=1.$ Moreover, by the compactness of the trace map in \autoref{lem:korn_compact_trace}, we can also assume
$$
T_\Gamma z_j \to  T_\Gamma z
\qquad \text{strongly in } \rL^2(\Gamma).
$$
Using \autoref{cond:dense_boundary_observations}, and in particular the
consequence (see \eqref{eq:observation norm_convergence})
$$
f_j \to f \text{ in } \rL^2(\Gamma), \ \delta_j \to0 \quad \Longrightarrow \quad
\| P_{\delta_j} f_j \|_{Y_{\delta_j}}
\to \| f \|_{\rL^2(\Gamma)},
$$
with $f_j=T_\Gamma z_j$ and $f=T_\Gamma z$, we obtain $ \| P_{\delta_j} T_\Gamma z_j \|_{ Y_{\delta_j}}
\to \| T_\Gamma z\|_{\rL^2(\Gamma)}.$
By weak lower semi-continuity of $a_\alpha$ we get
$$
a_\alpha(z,z) \leq \liminf_{j\to\infty} a_\alpha(z_j,z_j).
$$
Therefore,
\begin{equation*}
\begin{split}
a_{\mu,\Gamma}(z,z) &= a_\alpha( z ,z) + \mu \| T_\Gamma z \|_{ \rL^2( \Gamma)}^2\leq \liminf_{ j \to \infty} a_\alpha( z_j, z_j) +
\mu \lim_{ j \to \infty} \|P_{ \delta_j } T_\Gamma z_j \|_{Y_{\delta_j}}^2
\\ &\leq \liminf_{j\to\infty} \left(
a_\alpha(z_j,z_j) + \mu
\|P_{\delta_j}T_\Gamma z_j\|_{Y_{\delta_j}}^2
\right) = \liminf_{j\to\infty} a_{ \mu, \delta_j}( z_j, z_j) .
\end{split}
\end{equation*}
Since $\| z \|_{\rL^2(\Omega)}=1$, this gives
$$
\kappa_{\mu,\Gamma}
\leq
a_{\mu,\Gamma}(z,z)
\leq
\liminf_{j\to\infty}a_{\mu,\delta_j}(z_j,z_j)
\leq
\liminf_{j\to\infty}\left(\kappa_{\mu,\delta_j}+\frac1j\right).
$$
In conclusion $ \kappa_{\mu, \Gamma} \leq \liminf_{ \delta \to 0^+ }\kappa_{ \mu, \delta}$.

\emph{Step $2$: $\limsup_{\delta\to0^+}\kappa_{\mu,\delta}
\leq
\kappa_{\mu,\Gamma}.$} Let $\varepsilon>0$. By the definition of
$\kappa_{\mu,\Gamma}$, there exists $z_\varepsilon\in V$ such that
$$
\| z_\varepsilon\|_{ \rL^2 ( \Omega)} = 1, \qquad a_{\mu,\Gamma}(z_\varepsilon , z_\varepsilon) \leq \kappa_{ \mu, \Gamma} + \varepsilon.
$$
For this fixed function $z_\varepsilon$, \autoref{cond:dense_boundary_observations}
implies
$$
\| P_\delta T_\Gamma z_\varepsilon \|_{Y_\delta}^2 \to \| T_\Gamma  z_\varepsilon \|_{ \rL^2( \Gamma) }^2
\qquad \text{as }\delta \to 0^+.
$$
Hence
$$
a_{\mu,\delta}(z_\varepsilon,z_\varepsilon) = a_{\alpha} (z_\varepsilon, z_\varepsilon) + \mu \| P_{\delta} T_{\Gamma} z_{\varepsilon} \|_{Y_{\delta}}^2
\to  a_{\alpha} (z_\varepsilon, z_\varepsilon) + \mu  \| T_\Gamma  z_\varepsilon \|_{ \rL^2( \Gamma) }^2 = 
a_{\mu,\Gamma}(z_\varepsilon,z_\varepsilon).
$$
By definition of
$\kappa_{\mu,\delta}$, we obtain
$$
\kappa_{\mu,\delta} \leq a_{\mu,\delta}(z_\varepsilon,z_\varepsilon).
$$
In conclusion
$$
\limsup_{ \delta \to0^ + } \kappa_{ \mu, \delta} \leq a_{ \mu, \Gamma}(z_\varepsilon , z_\varepsilon) \leq
\kappa_{ \mu, \Gamma} + \varepsilon.
$$
Since $\varepsilon>0$ is arbitrary, we have proved $
\limsup_{ \delta \to0^ + } \kappa_{ \mu, \delta} \leq \kappa_{\mu,\Gamma}.$
\end{proof}

\subsection{Feedback-induced coercivity}
\noindent 
We can now combine the full-trace approximation with the discrete
convergence of the observations to obtain the coercivity estimate needed later. This proves the second part of \autoref{thm:main_boundary_feedback_spectral_gaps}.
\begin{prop}[Feedback-induced coercivity]
\label{prop:feedback_induced_coercivity}
Assume \autoref{cond:dense_boundary_observations}. Let
$0<\kappa<\kappa_{\infty,\Gamma}$. Then there exist $\mu_0>0$ and
$\delta_0>0$ such that, for every $\mu\geq\mu_0$ and every
$0<\delta\leq\delta_0$,
\begin{equation}
a_{\mu,\delta}(z,z) \geq \kappa \| z\|_{ \rL^2(\Omega)}^2
\qquad \text{for all }z \in V.
\label{eq: feedback induced coercivity}
\end{equation}
\end{prop}
\begin{proof}
By \autoref{lem:continuous_boundary_penalization_limit}, we can choose
$\mu_0>0$ such that
$$
\kappa_{\mu_0,\Gamma}>\kappa.
$$
Then \autoref{lem:discrete_gap_convergence} gives $\delta_0>0$ such that
$$
\kappa_{\mu_0,\delta} \geq \kappa
 \qquad    \text{for all } 0 < \delta \leq \delta_0 .
$$
If $\mu\geq\mu_0$, then
$$
a_{\mu,\delta}(z,z) \geq a_{\mu_0,\delta} (z , z) \qquad \text{for all } z \in V,
$$
and the estimate \eqref{eq: feedback induced coercivity} follows.
\end{proof}

\begin{rem}
The number $\kappa_{\infty,\Gamma}$ is the maximal spectral gap attainable by
this boundary-feedback mechanism. Indeed, for every $z\in V$ satisfying
$T_\Gamma z=0$,
\[
a_{\mu,\delta}(z,z)=a_\alpha(z,z),
\]
and therefore
\[
\kappa_{\mu,\delta}
\leq
\inf_{\substack{z\in V\setminus\{0\}\\T_\Gamma z=0}}
\frac{a_\alpha(z,z)}{\|z\|_{\rL^2(\Omega)}^2}
=
\kappa_{\infty,\Gamma}.
\]
Moreover, \autoref{lem:continuous_boundary_penalization_limit} and
\autoref{lem:discrete_gap_convergence} give the iterated limit
\[
\lim_{\mu\to\infty}\,
\lim_{\delta\to0^+}
\kappa_{\mu,\delta}
=
\kappa_{\infty,\Gamma}.
\]
Thus, by first choosing the feedback strength sufficiently large and then
taking sufficiently fine observations, the actual feedback gap can be made
arbitrarily close to, but cannot exceed, the limiting mixed-boundary gap.
\end{rem}

\section{Exponential synchronisation and concrete regimes}
\label{sec:synchronisation}
\noindent 
We now combine the coercivity generated by the boundary feedback from
\autoref{prop:feedback_induced_coercivity} with the non-linear energy balance
for the error equation in \eqref{eq: error assimilation equality distribution}.
The feedback term provides linear damping of the error, whereas the non-linear
term is controlled by the long-time averaged symmetric-gradient energy of the reference solution.
We first formulate this abstract argument and then give concrete sufficient
conditions under which the required averaged smallness condition is satisfied.

\subsection{Abstract synchronisation theorem}
\noindent
Let $C_4>0$ be a constant in the two-dimensional Ladyzhenskaya inequality
\begin{equation}
\| z \|_{ \rL^4( \Omega) }^2 \leq
C_4 \| z \|_{ \rL^2(\Omega)} \| z \|_{ \rH^1( \Omega)}
 \qquad \text{for all } z \in \rH^1( \Omega; \R^2).
\label{eq:ladyzhenskaya}
\end{equation}
For $\kappa>0$ and $\nu>0$, we recall that in \eqref{eq:def_theta_kappa_nu_intro} we have introduced
\begin{equation*}
\Theta_{\kappa,\nu}
=
C_4^2 C_{\mathrm{Korn}}
\left(\kappa^{-1}+(2\nu)^{-1}\right),
\end{equation*}
where $C_{\mathrm{Korn}}$ is the viscosity-independent constant in Korn inequality, see
\eqref{eq:korn_second_squared} in \autoref{lem:korn_compact_trace}.

We start with a basic non-linear bound for the error equation. The only
point requiring some care is that the Navier-slip form $a_\alpha$ may fail to
control the full $\rH^1$ norm when undamped rigid motions are present. The
feedback-induced coercivity estimate supplies the missing $\rL^2$ control.

\begin{lem}[Non-linear error estimate]
\label{lem:nonlinear_error_estimate}
Assume that for some $\kappa>0$
\begin{equation}
a_{\mu,\delta} (z , z ) \geq
\kappa \| z \|_{ \rL^2(\Omega)}^2
 \qquad \text{ for all } z \in V.
\label{coercivity hp for synchronisation}
\end{equation}
Let $u$ be the solution of \eqref{eq: 2D nse} and let $w=v-u$ be the
assimilation error satisfying \eqref{eq: error_equation}. Then, in the sense of
distributions on $(0,\infty)$,
\begin{equation}
\label{eq:error_differential_ineq}
\frac{\d}{\d t} \| w(t) \|_{ \rL^2(\Omega)}^2
+ \left( \kappa- \Theta_{\kappa,\nu}\|D(u(t))\|_{ \rL^2(\Omega)}^2 \right)
\| w ( t ) \|_{ \rL^2(\Omega)}^2
\leq 0 .
\end{equation}
\end{lem}
\begin{proof}
The starting point is the energy identity
\eqref{eq: error assimilation equality distribution} from
\autoref{thm:wp_error_equation}, namely
$$
\frac{1}{2} \frac{\d}{\d t} \| w ( t ) \|_{ \rL^2(\Omega)}^2
+ a_{ \mu , \delta}( w(t) , w(t))
= - \int_\Omega (w(t) \cdot \nabla ) u( t ) \cdot w(t) \d x.
$$
Write
$$
\nabla u=D(u)+A(u),
\qquad
A(u):=\frac{\nabla u-(\nabla u)^T}{2}.
$$
Since $A(u)$ is antisymmetric, $w^TA(u)w=0$ pointwise. Therefore
$$
(w\cdot\nabla)u\cdot w
=
w^T\nabla u\,w
=
w^TD(u)w.
$$
By H\"older's inequality and \eqref{eq:ladyzhenskaya},
\begin{equation}
    \left| \int_\Omega (w\cdot\nabla) u \cdot w \d x \right|
\leq
\|D(u)\|_{\rL^2(\Omega)}\|w\|_{\rL^4(\Omega)}^2
\leq
C_4\|D(u)\|_{\rL^2(\Omega)}
\|w\|_{\rL^2(\Omega)}\|w\|_{\rH^1(\Omega)}.
\label{eq: bound non_linear term lemma}
\end{equation}
Moreover, by \eqref{eq:korn_second_squared},
$a_{\mu,\delta}(w,w)\geq a_\alpha(w,w)\geq
2\nu\|D(w)\|_{\rL^2(\Omega)}^2$, and
\eqref{coercivity hp for synchronisation},
$$
\begin{aligned}
\|w\|_{\rH^1(\Omega)}^2
&\leq
C_{\mathrm{Korn}}
\left(
\|w\|_{\rL^2(\Omega)}^2
+
\|D(w)\|_{\rL^2(\Omega)}^2
\right) \leq
C_{\mathrm{Korn}}
\left(\kappa^{-1}+(2\nu)^{-1}\right)
a_{\mu,\delta}(w,w).
\end{aligned}
$$
Hence
$$
\left| \int_\Omega (w \cdot \nabla) u \cdot w\d x \right|
\leq
\Theta_{\kappa,\nu}^{1/2}
\|D(u)\|_{\rL^2(\Omega)}
\|w\|_{\rL^2(\Omega)}
a_{\mu,\delta}(w,w)^{1/2}.
$$
Young's inequality gives
$$
\left| \int_\Omega (w \cdot \nabla) u\cdot w \d x \right|
\leq
\frac{1}{2}a_{\mu,\delta}(w,w)
+
\frac{1}{2}\Theta_{\kappa,\nu}
\|D(u)\|_{\rL^2(\Omega)}^2
\|w\|_{\rL^2(\Omega)}^2.
$$
Substituting this estimate into the energy identity and using
\eqref{coercivity hp for synchronisation} once more gives
\eqref{eq:error_differential_ineq}.
\end{proof}
\noindent
We recall that in \eqref{eq:def_MuD} we introduced the long-time average of the squared symmetric gradient
\begin{equation*}
M_u^D
:=
\limsup_{t\to\infty}
\frac{1}{t}
\int_0^t
\|D(u(s))\|_{\rL^2(\Omega)}^2\d s.
\end{equation*}
\begin{rem}[Only the symmetric gradient enters the error production]
The quantity $M_u^D$ is the natural non-linear threshold because the
antisymmetric part of $\nabla u$ does not contribute to the error energy. In
particular, a rigid rotation may satisfy
$$
D(u)=0 \qquad \text{while} \qquad
\nabla u\neq0.
$$
For example, if $\alpha=0$ and $\Omega$ is a disk or a concentric annulus, the
field $u(x)=\omega x^\perp$ is compatible with the impermeability and
perfect-slip boundary conditions. It contributes nothing to
$\int_\Omega (w\cdot\nabla)u\cdot w\d x$, although a criterion based on
$\|\nabla u\|_{\rL^2}^2$ would count it as destabilising.
\end{rem}
\noindent 
We are now in a position to prove our main synchronisation result from \autoref{thm:synchronisation}. We recall, that it applies when the feedback
damping is stronger, on average, than the term generated by the symmetric gradient.
\begin{proof}[Proof of \autoref{thm:synchronisation}]
Since $0<\kappa < \kappa_{ \infty, \Gamma}$,
\autoref{prop:feedback_induced_coercivity} gives $\mu_0>0$ and $\delta_0>0$
such that
$$
a_{\mu,\delta}(z,z) \geq \kappa \| z \|_{ \rL^2(\Omega) }^2
\qquad \text{for all } z \in V
$$
provided that $\mu \geq \mu_0$ and $0<\delta\leq\delta_0$. Therefore
\autoref{lem:nonlinear_error_estimate} applies. Taking
$$
E(t)=\|w(t)\|_{\rL^2(\Omega)}^2,
\qquad
q(t)=\|D(u(t))\|_{\rL^2(\Omega)}^2,
\qquad
a=\kappa,
\qquad
b=\Theta_{\kappa,\nu},
$$
the condition $\Theta_{\kappa,\nu}M_u^D<\kappa$ is exactly
\eqref{eq: abstract averaged damping hp}. The conclusion follows from
\autoref{lem: modified gronwall}, after taking square roots.
\end{proof}
\noindent 
We next deduce \autoref{cor:unforced_synchronisation}, which gives synchronisation for the unforced case $g = 0$. Its proof is based on showing that, thanks to $g = 0$, the contribution given by the symmetric gradient vanishes when we compute the long-time average.
\begin{proof}[Proof of \autoref{cor:unforced_synchronisation}]
The energy inequality for the reference solution gives
$$
\frac{1}{2} \| u(T) \|_{\rL^2(\Omega)}^2 +
\int_0^T a_\alpha ( u (t) , u(t))\d t \leq 
\frac{1}{2} \| u_0 \|_{ \rL^2(\Omega)}^2.
$$
Since $a_\alpha(u,u)\geq2\nu\|D(u)\|_{\rL^2(\Omega)}^2$, it follows that
$$
\frac{1}{T}\int_0^T\|D(u(t))\|_{\rL^2(\Omega)}^2\d t
\leq
\frac{\|u_0\|_{\rL^2(\Omega)}^2}{4\nu T}
\to 0.
$$
Thus $M_u^D=0$, and the non-linear gap condition for the symmetric gradient term holds for every
$\kappa\in(0,\kappa_{\infty,\Gamma})$.
\end{proof}

\subsection{Large-viscosity regimes}
\noindent 
We now discuss two situations in which the condition \eqref{eq:non-linear_gap_condition} can be
ensured by taking the viscosity sufficiently large. The first covers the case
where viscosity itself has no rigid-motion obstruction. The second covers the
case $\alpha>0$, where the boundary friction damps rigid motions.

\subsubsection*{Large viscosity without tangential rigid motions}
\noindent 
The following geometric condition excludes rigid motions compatible with the
impermeability constraint
\begin{equation}
\label{eq:no_tangential_rigid_motions}
\left \lbrace  z(x) = a + \omega x^\perp:
a \in\mathbb{R}^2, \, \omega \in \R, \, z \cdot n=0 \text{ on } \partial \Omega
\right \rbrace = \{ 0 \}.
\end{equation}
\noindent 
Condition \eqref{eq:no_tangential_rigid_motions} removes the kernel of the symmetric gradient on $V$ and therefore yields the Korn-Poincar\'e estimate used in the following result.
\begin{prop}[Large viscosity without tangential rigid motions]
\label{prop:large_viscosity_synchronisation}
Assume \autoref{cond:dense_boundary_observations} and
\eqref{eq:no_tangential_rigid_motions}. Let $C_\Omega$ be the constant in
\eqref{eq:korn_poincare_from_no_rigid_motions}. If
\begin{equation}
\nu^4 >\frac{C_4^2 C_\Omega^3}{8} \|g\|_{\rL^2(\Omega)}^2,
\label{eq: large viscosity hp}
\end{equation}
then synchronisation holds in the following sense: for this fixed $\nu$, there
exist $\mu_0=\mu_0(\nu)>0$ and $\delta_0=\delta_0(\nu)>0$ such that the
conclusion of \autoref{thm:synchronisation} holds for every $\mu\geq\mu_0$ and
every $0<\delta\leq\delta_0$.
\end{prop}
\begin{proof}
Since $\alpha\geq0$,
$$
a_\alpha(z,z)
=
2\nu\|D(z)\|_{\rL^2(\Omega)}^2
+
\alpha\|z_\tau\|_{\rL^2(\partial\Omega)}^2
\geq
2\nu\|D(z)\|_{\rL^2(\Omega)}^2.
$$
Using \eqref{eq:korn_poincare_from_no_rigid_motions}, we obtain
$$
a_\alpha(z,z)
\geq
\frac{2\nu}{C_\Omega}\|z\|_{\rL^2(\Omega)}^2
\qquad\text{for all }z\in V.
$$
Thus $\kappa_{\infty,\Gamma}\geq2\nu/C_\Omega$. We choose
$$
\kappa_\nu:=\frac{\nu}{C_\Omega}.
$$
Then $0<\kappa_\nu<\kappa_{\infty,\Gamma}$, and
\autoref{prop:feedback_induced_coercivity} gives, for this fixed $\nu$,
constants $\mu_0(\nu)>0$ and $\delta_0(\nu)>0$ such that
$$
a_{\mu,\delta}(z,z)
\geq
\kappa_\nu\|z\|_{\rL^2(\Omega)}^2
\qquad\text{for all }z\in V
$$
whenever $\mu\geq\mu_0(\nu)$ and $0<\delta\leq\delta_0(\nu)$.

We next estimate the reference solution. Testing \eqref{eq: 2D nse}$_1$ with
$u$, using \eqref{eq:korn_poincare_from_no_rigid_motions} and
$a_\alpha(u,u)\geq2\nu\|D(u)\|_{\rL^2}^2$, we get
$$
(g,u)_{\rL^2}
\leq
\|g\|_{\rL^2}\|u\|_{\rL^2}
\leq
C_\Omega^{1/2}\|g\|_{\rL^2}\|D(u)\|_{\rL^2}
\leq
\frac{1}{2}a_\alpha(u,u)
+
\frac{C_\Omega}{4\nu}\|g\|_{\rL^2}^2.
$$
Therefore
$$
\frac{\d}{\d t}\|u(t)\|_{\rL^2(\Omega)}^2
+
a_\alpha(u(t),u(t))
\leq
\frac{C_\Omega}{2\nu}\|g\|_{\rL^2(\Omega)}^2.
$$
It follows that
$$
\limsup_{t\to\infty}\frac{1}{t}
\int_0^t a_\alpha(u(s),u(s))\d s
\leq
\frac{C_\Omega}{2\nu}\|g\|_{\rL^2(\Omega)}^2.
$$
Since $a_\alpha(u,u)\geq2\nu\|D(u)\|_{\rL^2}^2$, we obtain
\begin{equation}
M_u^D
\leq
\frac{C_\Omega}{4\nu^2}\|g\|_{\rL^2(\Omega)}^2.
\label{eq:MuD_large_viscosity_bound_aux}
\end{equation}

It remains to use the viscosity-dependent control of the error. By
\eqref{eq:korn_poincare_from_no_rigid_motions},
\begin{equation}
\|w\|_{\rH^1(\Omega)}^2
\leq
C_\Omega\|D(w)\|_{\rL^2(\Omega)}^2
\leq
\frac{C_\Omega}{2\nu}a_{\mu,\delta}(w,w).
\label{eq: H1 bound with a mu delta}
\end{equation}
Repeating the same argument as in the proof of \autoref{lem:nonlinear_error_estimate}, we get to \eqref{eq: bound non_linear term lemma}, that is
$$
    \left| \int_\Omega (w\cdot\nabla) u \cdot w \d x \right|
\leq
\|D(u)\|_{\rL^2(\Omega)}\|w\|_{\rL^4(\Omega)}^2
\leq
C_4\|D(u)\|_{\rL^2(\Omega)}
\|w\|_{\rL^2(\Omega)}\|w\|_{\rH^1(\Omega)}.
$$
Applying to this last estimate the bound in \eqref{eq: H1 bound with a mu delta} and Young's inequality, we obtain
$$
\left|\int_\Omega(w\cdot\nabla)u\cdot w\d x\right|
\leq
\frac{1}{2}a_{\mu,\delta}(w,w)
+
\frac{C_4^2C_\Omega}{4\nu}
\|D(u)\|_{\rL^2(\Omega)}^2
\|w\|_{\rL^2(\Omega)}^2.
$$
The error energy identity and
$a_{\mu,\delta}(w,w)\geq\kappa_\nu\|w\|_{\rL^2}^2$ then give
$$
\frac{\d}{\d t}\|w(t)\|_{\rL^2(\Omega)}^2
+
\left(
\kappa_\nu
-
\frac{C_4^2C_\Omega}{2\nu}
\|D(u(t))\|_{\rL^2(\Omega)}^2
\right)
\|w(t)\|_{\rL^2(\Omega)}^2
\leq0.
$$
By \eqref{eq:MuD_large_viscosity_bound_aux}, the averaged damping condition is
satisfied if
$$
\frac{C_4^2C_\Omega}{2\nu}
\frac{C_\Omega}{4\nu^2}
\|g\|_{\rL^2(\Omega)}^2
<
\frac{\nu}{C_\Omega},
$$
which is exactly \eqref{eq: large viscosity hp}. The conclusion follows from
\autoref{lem: modified gronwall}.
\end{proof}
\begin{rem}
Assumption \eqref{eq:no_tangential_rigid_motions} says that viscosity damps
every admissible velocity field. It excludes the classical obstruction given by
rigid rotations. For instance, on a disk or on a concentric annulus, the field
$ x^\perp $ is tangent to the boundary and satisfies $D (x^\perp ) = 0$, so large
viscosity alone cannot damp that mode.
\end{rem}

\subsubsection*{Large viscosity with positive boundary friction}
\noindent 
We finally consider the case $\alpha>0$. In this regime the boundary friction
removes the rigid-motion obstruction, even in domains such as disks or
concentric annuli. This additional coercivity leads to the following large-viscosity synchronisation criterion.
\begin{prop}[Large viscosity with positive boundary friction]
\label{prop:large_viscosity_positive_alpha}
Assume \autoref{cond:dense_boundary_observations} and that $\alpha>0$, and fix $\underline{\nu}>0$. Let
$ \lambda_{ \alpha, \underline{\nu}}$ and $C_{\alpha, \underline{\nu}}$ be the
constants from \autoref{lem:positive_alpha_coercivity}. Choose any
$ \kappa \in( 0 , \lambda_{\alpha, \underline{\nu}})$. If
\begin{equation}
\nu >  \max \left \lbrace \frac{ C_4^2 C_{\alpha,\underline{\nu}}
}{ 2\kappa \lambda_{\alpha, \underline{\nu}}
}\| g \|_{\rL^2( \Omega)}^2, \underline{ \nu } \right \rbrace     ,
\label{eq:large_viscosity_positive_alpha_condition}
\end{equation}
then synchronisation holds in the following sense: for this fixed
$\nu$, there exist $\mu_0 = \mu_0(\nu)>0$ and
$\delta_0=\delta_0(\nu)>0$ such that the conclusion of
\autoref{thm:synchronisation} holds for every $\mu\geq\mu_0$ and every $0<\delta\leq\delta_0$.

In particular, taking
$\kappa=\lambda_{\alpha,\underline{\nu}}/2$, it is enough that
\begin{equation}
\nu
> \max \left \lbrace \frac{ C_4^2 C_{\alpha,\underline{\nu}}
}{ \lambda_{ \alpha , \underline{\nu}}^2
} \|g\|_{\rL^2(\Omega)}^2, \underline{ \nu} \right \rbrace   .
\label{eq:large_viscosity_positive_alpha_condition_simple}
\end{equation}
\end{prop}
\begin{proof}
For every $\nu\geq\underline\nu$,
$$
a_\alpha(z,z) = 2\nu\|D(z)\|_{\rL^2(\Omega)}^2
+ \alpha \| z_\tau\|_{ \rL^2(\partial\Omega) }^2
\geq 2 \underline{\nu} \| D(z) \|_{\rL^2(\Omega)}^2 + \alpha \| z_\tau \|_{\rL^2( \partial \Omega)}^2  .
$$
Therefore \eqref{eq:positive_alpha_spectral_gap} implies
$$
a_\alpha(z,z) \geq \lambda_{ \alpha, \underline{\nu}} \| z \|_{ \rL^2( \Omega)}^2 \qquad \text{for all }z\in V.
$$
In particular,
$$
\kappa_{\infty,\Gamma} \geq \lambda_{\alpha, \underline {\nu}}.
$$
Since $\kappa<\lambda_{\alpha,\underline{\nu}}$, then \autoref{prop:feedback_induced_coercivity} gives, for this fixed $\nu$, constants $\mu_0(\nu)>0$ and
$\delta_0(\nu)>0$ such that
$$
a_{\mu , \delta } ( z , z) \geq \kappa \| z \|_{\rL^2(\Omega)}^2 \qquad \text{for all } z \in V
$$
whenever $\mu \geq \mu_0(\nu)$ and $0 < \delta \leq \delta_0(\nu)$.

We use the symmetric-gradient representation of the non-linear term with the $\rH^1$ coercivity available in the positive-friction regime. By
\autoref{lem:nonlinear_error_estimate},
$$
(w\cdot\nabla)u\cdot w=w^TD(u)w.
$$
Therefore,
$$
\left|
\int_\Omega (w\cdot\nabla)u\cdot w\,\d x
\right| \leq
\int_\Omega |D(u)|\,|w|^2 \d x \leq \|D(u)\|_{\rL^2(\Omega)} \|w\|_{\rL^4(\Omega)}^2 .
$$
By Ladyzhenskaya's inequality and
\eqref{eq:positive_alpha_h1_coercivity}
$$
\|w\|_{\rL^4(\Omega)}^2 \leq C_4 \| w \|_{\rL^2(\Omega)} \| w \|_{ \rH^1(\Omega)} \leq C_4 C_{ \alpha, \underline{\nu}}^{1/2}
 \| w \|_{\rL^2(\Omega)} a_{\mu,\delta}(w,w)^{1/2}.
$$
Hence, by Young's inequality
$$
\left| \int_\Omega (w \cdot \nabla) u \cdot w \d x \right|
\leq \frac{1}{2} a_{\mu,\delta}(w,w) + \frac{1}{2} C_4^2 C_{\alpha,\underline{\nu}} \| D(u) \|_{\rL^2(\Omega)}^2 \| w \|_{\rL^2(\Omega)}^2 .
$$
Using the error energy identity \eqref{eq: error assimilation equality distribution} and the feedback-induced coercivity estimate \eqref{eq: feedback induced coercivity}, we get
\begin{equation}
\label{eq:error_positive_alpha_large_viscosity}
\frac{\d}{\d t} \| w ( t ) \|_{\rL^2(\Omega)}^2
+ \left( \kappa - C_4^2 C_{\alpha, \underline{\nu}} \| D ( u ( t ) ) \|_{\rL^2(\Omega)}^2 \right) \|w(t)\|_{\rL^2(\Omega)}^2
\leq 0  .
\end{equation}

It remains to estimate $M_u^D$. Testing \eqref{eq: 2D nse}$_1$ with $u$ and
using \eqref{eq:positive_alpha_spectral_gap}, we obtain
$$
(g,u)_{\rL^2}
\leq
\frac{1}{2\lambda_{\alpha,\underline{\nu}}}\|g\|_{\rL^2}^2
+
\frac{\lambda_{\alpha,\underline{\nu}}}{2}\|u\|_{\rL^2}^2
\leq
\frac{1}{2\lambda_{\alpha,\underline{\nu}}}\|g\|_{\rL^2}^2
+
\frac{1}{2}a_\alpha(u,u).
$$
Therefore
$$
\frac{\d}{\d t}\|u(t)\|_{\rL^2(\Omega)}^2
+
a_\alpha(u(t),u(t))
\leq
\frac{1}{\lambda_{\alpha,\underline\nu}}
\|g\|_{\rL^2(\Omega)}^2.
$$
Since $a_\alpha(u,u)\geq2\nu\|D(u)\|_{\rL^2(\Omega)}^2$, it follows that
\begin{equation}
M_u^D
\leq
\frac{1}{2\nu\lambda_{\alpha,\underline\nu}}
\|g\|_{\rL^2(\Omega)}^2.
\label{eq:MuD_positive_alpha_bound}
\end{equation}
Thus the averaged damping condition for
\eqref{eq:error_positive_alpha_large_viscosity} is
$$
C_4^2C_{\alpha,\underline\nu}
\frac{1}{2\nu\lambda_{\alpha,\underline\nu}}
\|g\|_{\rL^2(\Omega)}^2
<\kappa,
$$
which is \eqref{eq:large_viscosity_positive_alpha_condition}. The conclusion
follows from \autoref{lem: modified gronwall}.
\end{proof}
\begin{rem}
The coefficient $\alpha \geq 0$ is the Navier boundary friction coefficient.
The case $\alpha=0$ corresponds to perfect slip: the boundary prevents normal
penetration but does not slow down motion along the boundary. The case $\alpha>0$ corresponds to
partial slip with friction and produces the boundary dissipation
$$
\alpha \|u_\tau\|_{\rL^2(\partial\Omega)}^2 .
$$
This distinction is important on domains such as disks and concentric annuli.
The rigid rotation $x^\perp$ satisfies $D(x^\perp)=0$, so viscosity alone
does not damp it. However, when $\alpha>0$,
$$
\alpha\|  (x^\perp )_\tau \|_{ \rL^2( \partial \Omega)}^2>0,
$$
and the boundary friction removes this obstruction.
\end{rem}

\subsection{A small-forcing criterion}
\noindent
We now give a simple criterion which guarantees synchronisation when the
forcing is sufficiently small. The criterion assumes that the unnudged
Navier-slip dissipation $a_\alpha$ controls the $\rL^2$ norm. This spectral gap
is used only to estimate the long-time averaged symmetric-gradient energy $M_u^D$ in terms of the
forcing.

\begin{prop}[A small-forcing sufficient condition]
\label{prop:small_forcing_synchronisation}
Assume \autoref{cond:dense_boundary_observations}, and let $u$ be the reference
solution of \eqref{eq: 2D nse} with $g\in\rL^2(\Omega)$. Let
$0<\kappa<\kappa_{\infty,\Gamma}$ and let $\Theta_{\kappa,\nu}$ be defined by
\eqref{eq:def_theta_kappa_nu_intro}. Assume moreover that $a_\alpha$ has a spectral
gap: there exists $\lambda_\alpha>0$ such that
\begin{equation}
a_\alpha(z,z)
\geq
\lambda_\alpha\|z\|_{\rL^2(\Omega)}^2
\qquad\text{for all }z\in V.
\label{eq:alpha_spectral_gap}
\end{equation}
Then
\begin{equation}
M_u^D
\leq
\frac{1}{2\nu\lambda_\alpha}
\|g\|_{\rL^2(\Omega)}^2.
\label{eq:MuD_bound}
\end{equation}
Consequently, if
\begin{equation}
\|g\|_{\rL^2(\Omega)}^2
<
\frac{2\nu\lambda_\alpha\kappa}{\Theta_{\kappa,\nu}}
\qquad\text{for some }\kappa\in(0,\kappa_{\infty,\Gamma}),
\label{eq:small_forcing_condition}
\end{equation}
then the synchronisation conclusion of \autoref{thm:synchronisation} holds.
\end{prop}
\begin{proof}
Testing \eqref{eq: 2D nse}$_1$ with $u$ gives
$$
\frac{1}{2}\frac{\d}{\d t}\|u(t)\|_{\rL^2(\Omega)}^2
+
a_\alpha(u(t),u(t))
=
(g,u(t))_{\rL^2(\Omega)}.
$$
Using \eqref{eq:alpha_spectral_gap} and Young's inequality,
$$
(g,u)_{\rL^2(\Omega)}
\leq
\frac{1}{2\lambda_\alpha}\|g\|_{\rL^2(\Omega)}^2
+
\frac{\lambda_\alpha}{2}\|u\|_{\rL^2(\Omega)}^2
\leq
\frac{1}{2\lambda_\alpha}\|g\|_{\rL^2(\Omega)}^2
+
\frac{1}{2}a_\alpha(u,u).
$$
Therefore
$$
\frac{\d}{\d t}\|u(t)\|_{\rL^2(\Omega)}^2
+
a_\alpha(u(t),u(t))
\leq
\frac{1}{\lambda_\alpha}\|g\|_{\rL^2(\Omega)}^2.
$$
Integrating over $(0,T)$, dividing by $T$, and passing to the limsup gives
$$
\limsup_{T\to\infty}
\frac{1}{T}\int_0^T a_\alpha(u(s),u(s))\d s
\leq
\frac{1}{\lambda_\alpha}\|g\|_{\rL^2(\Omega)}^2.
$$
Since
$$
a_\alpha(u,u)
\geq
2\nu\|D(u)\|_{\rL^2(\Omega)}^2,
$$
we obtain \eqref{eq:MuD_bound}. Combining
\eqref{eq:small_forcing_condition} and \eqref{eq:MuD_bound},
$$
\Theta_{\kappa,\nu}M_u^D
\leq
\frac{\Theta_{\kappa,\nu}}{2\nu\lambda_\alpha}
\|g\|_{\rL^2(\Omega)}^2
<\kappa.
$$
Thus \eqref{eq:non-linear_gap_condition} holds, and the conclusion follows from
\autoref{thm:synchronisation}.
\end{proof}

\begin{rem}[Applicability of the small-forcing criterion]
The spectral-gap assumption \eqref{eq:alpha_spectral_gap} is a structural
coercivity property of $(\Omega,\alpha,\nu)$, not a smallness assumption on the
forcing. Its only role in \autoref{prop:small_forcing_synchronisation} is to
bound $M_u^D$ in terms of $\|g\|_{\rL^2}$.

The main cases are the following.
\begin{enumerate}
\item[(i)] If $\alpha>0$, boundary friction removes the rigid-motion kernel.
For every fixed $\underline\nu>0$ and every $\nu\geq\underline\nu$,
\autoref{lem:positive_alpha_coercivity} gives
$$
a_\alpha(z,z)
\geq
\lambda_{\alpha,\underline\nu}
\|z\|_{\rL^2(\Omega)}^2
\qquad\text{for all }z\in V.
$$
Thus \eqref{eq:alpha_spectral_gap} holds.

\item[(ii)] If $\alpha=0$ and
\eqref{eq:no_tangential_rigid_motions} holds, then
\autoref{lem:korn_poincare_no_rigid_motions} gives
$$
a_\alpha(z,z)
=
2\nu\|D(z)\|_{\rL^2(\Omega)}^2
\geq
\frac{2\nu}{C_\Omega}\|z\|_{\rL^2(\Omega)}^2,
$$
so one may take $\lambda_\alpha=2\nu/C_\Omega$.

\item[(iii)] If $\alpha = 0$, and $\mathcal{R}_{\Omega} \neq \{ 0 \}$, the un-nudged spectra-gap condition \eqref{eq:alpha_spectral_gap} fails. For example, on a disk or concentric annulus, the rigid rotation $z(x) = x^\perp$ satisfies
$$
z \cdot n = 0, \qquad D (z) = 0,
$$
and thus
$
a_0 (z,z) = 0.$ Therefore, \autoref{prop:small_forcing_synchronisation} cannot be applied because its auxiliary spectral-gap assumption fails. Nevertheless, the long-time averaged symmetric-gradient energy can be estimated after projecting away the rigid-motion component. The corresponding forced-flow criterion is given in \autoref{prop:perfect_slip_rigid_motions}. It includes the un-forced case and, more in general, forces of the form
$
g = r+ \nabla \phi , $ with $r \in \mathcal{R}_\Omega.
$

\end{enumerate}
\end{rem}

\subsection{Perfect slip in the presence of tangential rigid motions}
\label{subsec:perfect_slip_rigid_motions}

We finally consider the perfect-slip case $\alpha=0$ on a domain admitting
non-zero rigid motions compatible with the impermeability condition. This is
precisely the situation in which the unnudged dissipation has no positive
spectral gap. Nevertheless, the long-time averaged symmetric-gradient energy can be estimated
after separating the rigid and deformational components of the reference
solution.

Let
\begin{equation}
\mathcal R_\Omega
:=
\{r\in V:D(r)=0\},
\label{eq:rigid_motion_space}
\end{equation}
and let
\[
\Pi_{\mathcal R}:H\to\mathcal R_\Omega
\]
denote the $\rL^2(\Omega)$-orthogonal projection. We also denote by
\[
\mbbP_H:\rL^2(\Omega;\mathbb R^2)\to H
\]
the Helmholtz projection and set
\begin{equation}
g_\perp
:=
(I-\Pi_{\mathcal R})\mbbP_Hg.
\label{eq:nonrigid_solenoidal_force}
\end{equation}

\begin{prop}[Perfect slip with tangential rigid motions]
\label{prop:perfect_slip_rigid_motions}
Assume \autoref{cond:dense_boundary_observations}, let $\alpha=0$, and suppose
that $\mathcal R_\Omega\neq\{0\}$. Let $C_{\mathcal R}>0$ be the constant in
the projected Korn--Poincar\'e inequality
\eqref{eq:korn_poincare_modulo_rigid_motions}. Then the reference solution
satisfies
\begin{equation}
M_u^D
\leq
\frac{C_{\mathcal R}}{4\nu^2}
\|g_\perp\|_{\rL^2(\Omega)}^2.
\label{eq:MuD_rigid_motion_bound}
\end{equation}
Consequently, if for some
$\kappa\in(0,\kappa_{\infty,\Gamma})$,
\begin{equation}
\frac{C_{\mathcal R}\Theta_{\kappa,\nu}}{4\nu^2}
\|g_\perp\|_{\rL^2(\Omega)}^2
<\kappa,
\label{eq:rigid_motion_synchronisation_condition}
\end{equation}
then the synchronisation conclusion of
\autoref{thm:synchronisation} holds.
\end{prop}

\begin{proof}
Set
\[
u_{\mathcal R}:=\Pi_{\mathcal R}u,
\qquad
u_\perp:=(I-\Pi_{\mathcal R})u.
\]
For every $r\in\mathcal R_\Omega$, the matrix $\nabla r$ is antisymmetric.
Using $\nabla\cdot u=0$ and $u\cdot n=0$ on $\partial\Omega$, we obtain
\[
b(u,u,r)
=
-b(u,r,u)
=
-\int_\Omega u^T(\nabla r)u\,\d x
=
0.
\]
Moreover,
\[
a_0(u,r)=2\nu(D(u),D(r))_{\rL^2(\Omega)}=0.
\]
Testing the weak formulation of the reference equation by
$r\in\mathcal R_\Omega$ therefore gives
\[
\frac{\d}{\d t}(u(t),r)_{\rL^2(\Omega)}
=
(g,r)_{\rL^2(\Omega)}.
\]
Since $\mathcal R_\Omega$ is finite-dimensional, it follows that
\begin{equation}
u_{\mathcal R}(t)
=
\Pi_{\mathcal R}u_0
+
t\,\Pi_{\mathcal R}\mbbP_Hg.
\label{eq:rigid_component_evolution}
\end{equation}

In particular, the rigid component satisfies the exact energy identity
\[
\frac12\|u_{\mathcal R}(T)\|_{\rL^2(\Omega)}^2
=
\frac12\|u_{\mathcal R}(0)\|_{\rL^2(\Omega)}^2
+
\int_0^T(g,u_{\mathcal R}(t))_{\rL^2(\Omega)}\,\d t.
\]
Subtracting this identity from the energy inequality for $u$, and using the
orthogonal decomposition
$u=u_{\mathcal R}+u_\perp$, gives
\begin{equation}
\frac12\|u_\perp(T)\|_{\rL^2(\Omega)}^2
+
2\nu\int_0^T\|D(u(t))\|_{\rL^2(\Omega)}^2\,\d t
\leq
\frac12\|u_\perp(0)\|_{\rL^2(\Omega)}^2
+
\int_0^T(g_\perp,u_\perp(t))_{\rL^2(\Omega)}\,\d t.
\label{eq:projected_reference_energy}
\end{equation}
Here we have used
\[
D(u_\perp)=D(u),
\]
because $D(u_{\mathcal R})=0$.

Since $u_\perp(t)$ is orthogonal to $\mathcal R_\Omega$, the projected
Korn--Poincar\'e inequality gives
\[
\|u_\perp(t)\|_{\rL^2(\Omega)}
\leq
C_{\mathcal R}^{1/2}
\|D(u(t))\|_{\rL^2(\Omega)}.
\]
Consequently,
\[
(g_\perp,u_\perp)
\leq
C_{\mathcal R}^{1/2}
\|g_\perp\|_{\rL^2(\Omega)}
\|D(u)\|_{\rL^2(\Omega)}
\leq
\nu\|D(u)\|_{\rL^2(\Omega)}^2
+
\frac{C_{\mathcal R}}{4\nu}
\|g_\perp\|_{\rL^2(\Omega)}^2.
\]
Substituting this estimate into
\eqref{eq:projected_reference_energy}, dividing by $\nu T$, and discarding
the non-negative terminal term, we obtain
\[
\frac1T\int_0^T
\|D(u(t))\|_{\rL^2(\Omega)}^2\,\d t
\leq
\frac{\|u_\perp(0)\|_{\rL^2(\Omega)}^2}{2\nu T}
+
\frac{C_{\mathcal R}}{4\nu^2}
\|g_\perp\|_{\rL^2(\Omega)}^2.
\]
Passing to the limsup proves
\eqref{eq:MuD_rigid_motion_bound}. Condition
\eqref{eq:rigid_motion_synchronisation_condition} then implies
\[
\Theta_{\kappa,\nu}M_u^D<\kappa,
\]
and the conclusion follows from \autoref{thm:synchronisation}.
\end{proof}

\begin{rem}[Rigid and gradient forcing]
\label{rem:rigid_gradient_forcing}
If
\[
g=r+\nabla\phi,
\qquad
r\in\mathcal R_\Omega,
\quad
\phi\in\rH^1(\Omega),
\]
then $\mbbP_H\nabla\phi=0$ and $\mbbP_Hr=r$. Indeed, gradients are
$\rL^2$-orthogonal to $H$ because its elements are divergence-free and
have vanishing normal trace. Hence $\mbbP_Hg=r$ and $g_\perp=0$. Therefore
\[
M_u^D=0.
\]
The rigid component of the reference solution may grow linearly according to
\eqref{eq:rigid_component_evolution}, but its symmetric gradient vanishes. This also shows why the boundary feedback is essential for
synchronisation from arbitrary initial data in this regime. For
every non-zero $r\in\mathcal R_\Omega$,
\[
a_0(r,r)=0,
\]
so the natural viscous dissipation does not damp an error in the rigid
direction. On the other hand, $T_\Gamma r\neq0$, since a non-zero planar rigid
motion cannot vanish on a non-trivial boundary arc. Thus sufficiently resolved
boundary feedback detects and damps precisely the directions missed by the
un-nudged system.
\end{rem}

\appendix 
\section{Auxiliary analytical results}
\label{sec: appendix}
\subsection{Korn, trace, and rigid-motions}

This subsection collects the functional-analytic estimates used in the spectral-gap and synchronisation arguments. We first recall Korn's second inequality and the compactness of the tangential trace. We then prove Korn-Poincar\'e estimates obtained from a partial Dirichlet condition or from the exclusion of admissible rigid motions, and finally present the coercivity provided by positive boundary friction.

\begin{lem}[Korn and compact trace estimates]
\label{lem:korn_compact_trace}
There exists a constant $C_{\mathrm{Korn}}=C_{\mathrm{Korn}}(\Omega)>0$ such that
\begin{equation}
\|z\|_{\rH^1(\Omega)}^2
\leq C_{\mathrm{Korn}}
\left(
\|z\|_{\rL^2(\Omega)}^2
+
\|D(z)\|_{\rL^2(\Omega)}^2
\right)
\qquad\text{for all }z\in \rH^1(\Omega;\R^2).
\label{eq:korn_second_squared}
\end{equation}
Consequently, there is a constant $C_K=C_K(\Omega,\nu)>0$ such that
\begin{equation}
\| z \|_{\rH^1(\Omega)}^2 \leq C_K \left( \|z\|_{\rL^2(\Omega)}^2+ a_\alpha(z,z) \right)
\qquad  \text{for all } z \in V .
\label{eq: korn inequality}
\end{equation}
Moreover, the tangential trace map
\begin{equation*}
T_\Gamma : V \to \rL^2(\Gamma), \qquad T_\Gamma z = z_\tau|_\Gamma,
\end{equation*}
is compact.
\end{lem}
\begin{proof}
Korn's second inequality gives
\begin{equation}
\|z\|_{\rH^1(\Omega)} \leq
C( \| z \|_{ \rL^2( \Omega ) }+\|D(z)\|_{ \rL^2(\Omega)}),
\qquad z \in \rH^1 ( \Omega ; \R^2),
\label{eq: korn_second_inequality}
\end{equation}
see \cite[Chapter 1.2.2]{Oleinik1992}. Squaring and changing the constant gives
\eqref{eq:korn_second_squared}. Since $\nu>0$ and $\alpha\geq0$,
$$
2 \nu \| D ( z ) \|_{ \rL^2( \Omega ) }^2 \leq a_\alpha ( z , z),
$$
so \eqref{eq: korn inequality} follows, for instance with
$$
C_K=C_{\mathrm{Korn}}\max\left\{1,(2\nu)^{-1}\right\}.
$$

For the second part of the statement, the trace theorem gives a continuous map
$$
\rH^1( \Omega; \R^2) \to \rH^{1/2}( \partial \Omega; \R^2),
$$
and the embedding
$$
\rH^{ 1 / 2 } (\partial\Omega) \hookrightarrow \rL^2( \partial \Omega)
$$
is compact because $\partial\Omega$ is compact and one-dimensional.
Restricting the trace to the relatively open subset $\Gamma$ and taking the
tangential component are continuous operations from $\rL^2( \partial \Omega ; \R^2)$ to
$\rL^2(\Gamma)$. Hence $T_\Gamma : V\to \rL^2( \Gamma )$ is compact.
\end{proof}
\noindent 
For every $z\in V$, the boundary condition $z\cdot n=0$ already sets the
normal component to zero on $\partial\Omega$. If we also impose
$T_\Gamma z=0$, then the tangential component vanishes on $\Gamma$, and
therefore $z=0$ on this boundary part. Since a non-zero rigid motion cannot
vanish on a non-trivial boundary arc, the partial Dirichlet condition removes
the rigid-motion kernel and leads to the Korn-Poincar\'e inequality below.
\begin{lem}[Korn-Poincar\'e inequality with a partial Dirichlet condition]
\label{lem:korn_poincare_partial_dirichlet}
Set
$$
V_\Gamma:=\{z\in V:T_\Gamma z=0\}.
$$
Then there exists a constant $C_{\mathrm{KP},\Gamma} = C_{\mathrm{KP},\Gamma}(\Omega, \Gamma)>0$ such that
\begin{equation}
\|z\|_{\rH^1(\Omega)}^2 \leq C_{\mathrm{KP},\Gamma} \|D(z)\|_{\rL^2(\Omega)}^2
\qquad \text{for all }z \in V_\Gamma.
\label{eq:korn_poincare_partial_dirichlet}
\end{equation}
\end{lem}

\begin{proof}
Since every $z\in V$ satisfies $z\cdot n=0$ on $\partial\Omega$, the
additional condition $T_\Gamma z=0$ implies that the full trace of $z$
vanishes on $\Gamma$.
Suppose, by contradiction, that
\eqref{eq:korn_poincare_partial_dirichlet} is false. Then there exists a
sequence $(z_j)\subset V_\Gamma$ such that
$$
\|z_j\|_{\rH^1(\Omega)}=1,
\qquad \|D(z_j)\|_{\rL^2(\Omega)}\to0.
$$
The sequence $(z_j)$ is bounded in $\rH^1(\Omega;\R^2)$. Hence, by the compact
embedding $\rH^1(\Omega)\hookrightarrow \rL^2(\Omega)$, after passing to a
subsequence we have
$$
z_j\rightharpoonup z
\quad\text{ in }\rH^1(\Omega;\mathbb{R}^2),
\qquad
z_j\to z
\quad \text{ in }\rL^2(\Omega;\mathbb{R}^2).
$$
Applying Korn's second inequality \eqref{eq: korn_second_inequality} to
$z_j-z_k$, we obtain
$$
\|z_j-z_k\|_{\rH^1(\Omega)}
\leq C\left(
\|z_j-z_k\|_{\rL^2(\Omega)} + \|D(z_j-z_k)\|_{\rL^2(\Omega)}
\right).
$$
The first term on the right-hand side tends to zero as $j,k\to\infty$ because
$(z_j)$ is Cauchy in $\rL^2(\Omega)$. Moreover,
$$
\|D(z_j-z_k)\|_{\rL^2(\Omega)}
\leq
\|D(z_j)\|_{\rL^2(\Omega)}
+
\|D(z_k)\|_{\rL^2(\Omega)}
\to0.
$$
Therefore $(z_j)$ is Cauchy in $\rH^1(\Omega;\R^2)$, and hence
$$
z_j\to z
\qquad
\text{ in }\rH^1(\Omega;\R^2).
$$
It follows that
$$
\|z \|_{\rH^1 (\Omega)}=1,
\qquad D(z)=0.
$$
By continuity of the trace operator, the full trace of $z$ also vanishes on
$\Gamma$.
Since $\Omega$ is connected and $D(z)=0$, the vector field $z$ is an
infinitesimal rigid motion. In two dimensions, it has the form
$$
z(x)=a+b x^\perp,
\qquad a\in \R^2,
\quad b \in \mathbb{R},
\quad x^\perp=(-x_2,x_1).
$$
The relatively open set $\Gamma$ contains a nontrivial boundary arc. Since
$z$ is continuous and its trace vanishes almost everywhere on $\Gamma$, it
vanishes on this entire arc. On the other hand, a nonzero rigid motion
$x \mapsto a+b x^\perp$ can vanish at most at one point. Therefore $z\equiv0$,
which contradicts $\|z\|_{\rH^1(\Omega)}=1$. This proves
\eqref{eq:korn_poincare_partial_dirichlet}.
\end{proof}
\noindent 
\begin{lem}[Korn--Poincar\'e inequality modulo rigid motions]
\label{lem:korn_poincare_modulo_rigid_motions}
Let $\mathcal R_\Omega$ be defined by
\eqref{eq:rigid_motion_space}, and let
$\Pi_{\mathcal R}$ be the $\rL^2(\Omega)$-orthogonal projection onto
$\mathcal R_\Omega$. Then there exists a constant
$C_{\mathcal R}=C_{\mathcal R}(\Omega)>0$ such that
\begin{equation}
\|z\|_{\rL^2(\Omega)}^2
+
\|\nabla z\|_{\rL^2(\Omega)}^2
\leq
C_{\mathcal R}\|D(z)\|_{\rL^2(\Omega)}^2
\label{eq:korn_poincare_modulo_rigid_motions}
\end{equation}
for every $z\in V$ satisfying $\Pi_{\mathcal R}z=0$.
\end{lem}

\begin{proof}
Suppose that the estimate is false. Then there exists a sequence
$(z_j)\subset V$ such that
\[
\Pi_{\mathcal R}z_j=0,
\qquad
\|z_j\|_{\rH^1(\Omega)}=1,
\qquad
\|D(z_j)\|_{\rL^2(\Omega)}\to0.
\]
By compactness, after passing to a subsequence,
\[
z_j\to z
\quad\text{strongly in }\rL^2(\Omega),
\qquad
z_j\rightharpoonup z
\quad\text{weakly in }\rH^1(\Omega).
\]
Korn's second inequality, applied to $z_j-z_k$, shows as in the proof of
\autoref{lem:korn_poincare_partial_dirichlet} that
\[
z_j\to z
\qquad\text{strongly in }\rH^1(\Omega).
\]
Hence
\[
\|z\|_{\rH^1(\Omega)}=1,
\qquad
D(z)=0.
\]
Thus $z\in\mathcal R_\Omega$. On the other hand, the strong
$\rL^2$ convergence and $\Pi_{\mathcal R}z_j=0$ imply
$\Pi_{\mathcal R}z=0$. Therefore
\[
z\in\mathcal R_\Omega\cap\mathcal R_\Omega^\perp=\{0\},
\]
contradicting $\|z\|_{\rH^1(\Omega)}=1$.
\end{proof}
A different way to obtain Korn-Poincar\'e  coercivity is to assume directly that the domain admits no non-zero rigid motion satisfying the impermeability condition. In this case, the estimate holds on the whole space $V$, without imposing a partial Dirichlet condition.

\begin{lem}[Korn-Poincar\'e inequality without rigid motions]
\label{lem:korn_poincare_no_rigid_motions}
Assume \eqref{eq:no_tangential_rigid_motions}. Then there exists $C_\Omega > 0$ such that
\begin{equation}
\|z\|_{\rL^2(\Omega)}^2 + \| \nabla z \|_{ \rL^2(\Omega)}^2 \leq C_\Omega \| D( z ) \|_{\rL^2(\Omega)}^2 \qquad \text{for all } z \in V .
\label{eq:korn_poincare_from_no_rigid_motions}
\end{equation}
\end{lem}
\begin{proof}
Suppose that \eqref{eq:korn_poincare_from_no_rigid_motions} fails. Then there
exists a sequence $z_j \in V$ such that
$$
\| z_j \|_{ \rL^2( \Omega)}^2 + \| \nabla z_j \|_{ \rL^2( \Omega) }^2 = 1,
\qquad \| D ( z_j ) \|_{ \rL^2 ( \Omega ) } \to 0 .
$$
In particular, $(z_j)$ is bounded in $ \rH^1 ( \Omega ; \mathbb{R}^2 ) $. By the
compact embedding $ \rH^1( \Omega) \hookrightarrow \rL^2(\Omega)$, after passing
to a subsequence, we have
$$
z_j \to z \quad\text{ in } \rL^2 ( \Omega ; \mathbb{R}^2), \qquad z_j  \rightharpoonup  z  \quad \text{ in } \rH^1( \Omega; \mathbb{R}^2).
$$
We now apply Korn's inequality \eqref{eq: korn_second_inequality} to $z_j-z_k$, and we obtain
$$
\| z_j-z_k \|_{\rH^1} \leq C (
\| z_j- z_k\|_{\rL^2} + \| D ( z_j ) - D( z_k ) \|_{\rL^2} ).
$$
The first term tends to zero because $z_j$ is Cauchy in $\rL^2$, and the second
term tends to zero because $ \| D(z_j) \|_{\rL^2} \to 0$. Hence $(z_j)$ is Cauchy in
$ \rH^1 $, and therefore
$$
z_j \to z \quad \text{strongly in } \rH^1( \Omega ; \mathbb{R}^2).
$$
Since $V$ is closed in $ \rH^1$, we have $z \in V$, and by strong convergence
$$
\| z \|_{ \rH^1 }=1, \qquad D(z)=0.
$$
Because $\Omega$ is connected, the condition $D(z)=0$ implies that $z$ is a
rigid motion, namely
$$
z(x)= a + \omega x^\perp
$$
for some $a \in \mathbb{R}^2$ and $\omega \in \mathbb{R}$. Since $z \in V$, its normal
trace satisfies $z \cdot n=0$ on $\partial \Omega$. Thus $z$ belongs to the class
of rigid motions excluded by \eqref{eq:no_tangential_rigid_motions}. Therefore
$z=0$, contradicting $ \| z \|_{ \rH^1}=1$.
\end{proof}
\noindent 
Positive boundary friction gives another way to remove the rigid-motion kernel. The symmetric-gradient term together with the tangential boundary term then controls both the $\rL^2$ norm and the full $\rH^1$ norm.
\begin{lem}[Coercivity with positive boundary friction]
\label{lem:positive_alpha_coercivity}
Assume that $\alpha>0$, and fix $\underline{\nu}>0$. Then there exist
constants $\lambda_{\alpha,\underline{\nu}}>0$ and $C_{\alpha,\underline{\nu}}>0$, depending only on
$\Omega$, $\alpha$, and $\underline\nu$, such that
\begin{equation}
2\underline{\nu} \| D(z) \|_{\rL^2(\Omega)}^2
+ \alpha \| z_\tau \|_{\rL^2(\partial\Omega)}^2
\geq \lambda_{\alpha,\underline{\nu}} \| z \|_{\rL^2(\Omega)}^2 \qquad \text{for all } z \in V,
\label{eq:positive_alpha_spectral_gap}
\end{equation}
and
\begin{equation}
\|z\|_{\rH^1(\Omega)}^2 \leq C_{\alpha,\underline\nu} \left(
2 \underline{\nu} \| D(z)\|_{\rL^2( \Omega)}^2
+ \alpha \| z_\tau \|_{\rL^2( \partial \Omega)}^2 \right) \qquad\text{for all }z\in V.
\label{eq:positive_alpha_h1_coercivity}
\end{equation}
\end{lem}
\begin{proof}
We first prove \eqref{eq:positive_alpha_spectral_gap}. Set
$$
Q(z) : = 2\underline \nu \|D(z) \|_{\rL^2 ( \Omega ) }^2 +  \alpha \| z_\tau \|_{ \rL^2 (\partial \Omega ) }^2,
\qquad z \in V .
$$
Since $\underline{\nu}>0$ and $\alpha>0$, if $Q(z)=0$ then
$$
D(z)=0 \quad \text{in } \Omega, \qquad z_\tau = 0 \quad \text{on } \partial \Omega .
$$
Moreover, because $z \in V$, we also have $z\cdot n=0$ on $\partial \Omega$ in the
trace sense. Hence $z=0$ on $\partial \Omega$. On the other hand, $D(z)=0$ in a
connected domain implies that $z$ is a rigid motion, namely
$$
z(x) = a  + \omega x^\perp
$$
for some $a \in \mathbb{R}^2$ and $\omega \in \mathbb{R}$. This rigid motion can
vanish on the whole boundary only if $ a = 0 $ and $\omega=0$. Therefore the kernel
of $Q$ in $V$ is trivial.

We now turn this kernel statement into a quantitative estimate. Suppose, by
contradiction, that \eqref{eq:positive_alpha_spectral_gap} is false. Then there
is a sequence $(z_j)_j \subset V$ such that
$$
\| z_j \|_{\rL^2( \Omega ) } = 1, \qquad Q ( z_j ) \to 0 .
$$
In particular, $D(z_j) \to 0$ in $\rL^2(\Omega)$ and
$z_{j,\tau} \to 0$ in $\rL^2(\partial \Omega)$. By Korn's second inequality
\eqref{eq: korn_second_inequality}, it holds
$$
\| z_j \|_{ \rH^1 ( \Omega ) } \leq C ( \| z_j \|_{ \rL^2( \Omega ) } + \| D ( z_j ) \|_{ \rL^2 ( \Omega ) }),
$$
so $(z_j)_j$ is bounded in $\rH^1(\Omega; \mathbb{R}^2)$. After passing to a
subsequence, using the compact embedding
$\rH^1(\Omega) \hookrightarrow \rL^2(\Omega)$, we can assume that
$$
z_j \to z \quad \text{ in }\rL^2(\Omega;\mathbb{R}^2), \qquad
z_j \rightharpoonup z \quad \text{ in } \rH^1( \Omega ; \mathbb{R}^2).
$$
The space $V$ is closed in $\rH^1$, hence $z \in V$. Passing to the limit in
the symmetric gradients gives $D(z)=0$. By continuity of the trace operator,
the weak convergence in $\rH^1(\Omega;\R^2)$ implies
$
z_{j,\tau}\rightharpoonup z_\tau$ in $\rH^{1/2}(\partial\Omega).
$
On the other hand,
$
z_{j,\tau}\to 0$ in $\rL^2(\partial\Omega).$ The uniqueness of the weak limit therefore gives
$z_\tau=0$ on $\partial \Omega$. Since $z \in V$, we also have $z \cdot n =0$ on
$\partial \Omega$. Thus $z$ belongs to the kernel of $Q$ in $V$, and the previous
paragraph gives $z=0$. This contradicts the strong $\rL^2$ convergence and the
normalisation $ \| z_j \|_{\rL^2(\Omega)}=1$. Therefore
\eqref{eq:positive_alpha_spectral_gap} holds for some
$\lambda_{\alpha,\underline\nu}>0$.

We now prove \eqref{eq:positive_alpha_h1_coercivity}. Combining Korn's second
inequality \eqref{eq: korn_second_inequality} with \eqref{eq:positive_alpha_spectral_gap}, we get
$$
\| z \|_{ \rH^1(\Omega ) }^2 \leq C ( \|z \|_{ \rL^2(\Omega) }^2+ \| D ( z ) \|_{ \rL^2(\Omega)}^2)
\leq C ( \lambda_{\alpha, \underline{\nu}}^{-1}
+( 2 \underline{\nu})^{-1}) Q(z)  .
$$
This is \eqref{eq:positive_alpha_h1_coercivity}, after renaming the constant on the right-hand side.
\end{proof}
\subsection{Well-posedness details}
\label{appendix: well posedness}
\noindent 
This subsection provides the proofs of the well-posedness results stated in \autoref{subsec: well posedness reference and DA system} for the reference Navier-Stokes system \eqref{eq: 2D nse} and for the boundary-feedback data-assimilation system \eqref{eq: da_model}. We first recall the standard two-dimensional argument for the reference problem and derive its energy estimate.

\begin{proof}[Proof of \autoref{lem: well_posedness_reference}]
The case $g=0$ is exactly the two-dimensional Navier-Stokes system investigated in \cite{Kelliher2006}, see \cite[Theorem 6.1]{Kelliher2006} for the well-posedness result. The inclusion of the forcing term $g\in \rL^2(\Omega; \mathbb{R}^2)$ follows
by the same Galerkin construction. Indeed, $g$ defines an element of
$\rL^2(0,T;V')$, and the Galerkin approximations satisfy the same compactness
bounds as in the unforced case, with the only additional term $(g,u^m)$ in the
energy estimate, where $u^m$ is a Galerkin approximant. 

In the following, we just formally show how to obtain the energy inequality \eqref{eq: energy inequality 2DNSE}. For a
smooth solution $u$, testing \eqref{eq: 2D nse}$_1$ by $u$
gives
$$
\frac{1}{2}\frac{\d}{\d t} \| u(t) \|_{\rL^2(\Omega)}^2 +(- \nu \Delta u(t) , u(t))_{ \rL^2(\Omega) } +b( u(t), u(t) , u(t)) =(g,u(t))_{\rL^2(\Omega)} .
$$
Since $\nabla \cdot u=0$ and $u \cdot n=0$ on $\partial \Omega$, we have
$b(u,u,u)=0$. Moreover, using
$-\nu \Delta u=-\div (2\nu D(u))$ and the Navier boundary condition in \eqref{eq: 2D nse}$_3$-\eqref{eq: 2D nse}$_4$, we obtain
\begin{equation*}
\begin{aligned}
(-\nu\Delta u,u)_{\rL^2(\Omega)}
&=
2\nu\int_\Omega |D(u)|^2 \d x
-\int_{\partial\Omega} (2\nu D(u)n)\cdot u \d\sigma  \\
&=
2\nu\int_\Omega |D(u)|^2\,\d x
+\alpha\int_{\partial\Omega}|(u)_\tau|^2 \d\sigma
=a_\alpha(u,u).
\end{aligned}
\end{equation*}
Therefore, at a formal level, we have
$$
\frac{1}{2}\frac{\d}{\d t}\| u  (t ) \|_{\rL^2(\Omega)}^2 +a_\alpha(u(t),u(t))
=(g,u(t))_{\rL^2(\Omega)} .
$$
Integrating in time gives
$$
\frac{1}{2} \| u(t) \|_{\rL^2(\Omega)}^2
+\int_0^t a_\alpha(u(s),u(s)) \d s = \frac{1}{2} \| u_0 \|_{\rL^2(\Omega)}^2
+\int_0^t (g,u(s))_{\rL^2(\Omega)} \d s .
$$
The energy identity holds at the Galerkin level. Passing to the weak limit and
using lower semicontinuity, the forcing term converges
by weak convergence, while the two non-negative quadratic terms on the left-hand
side are weakly lower semi-continuous. Hence the identity at the approximate
level gives the energy inequality \eqref{eq: energy inequality 2DNSE} for the weak limit.
\end{proof}
\noindent 
The same variational framework adopted in the previous proof also applies to the assimilated system \eqref{eq: da_model} because the feedback term defines a bounded, symmetric, and non-negative bilinear form on $V$.

\begin{proof}[Proof of \autoref{lem:wp_da_system}]
We write the problem in the abstract form
$$
v' +A_{ \mu , \delta} v + B( v , v) = F_\delta \qquad \text{in }  V',
$$
where
$$
\langle A_{ \mu, \delta} z , \phi \rangle_{ V', V} : = a_{ \mu, \delta}( z , \phi),
\qquad \langle B( z, z), \phi \rangle_{ V ' , V } :  = b( z , z , \phi),
$$
and
$$
\langle F_\delta ( t ), \phi \rangle_{ V' , V} : = ( g , \phi)_{\rL^2(\Omega) } + \mu \langle y_\delta( t ) , P_\delta T_\Gamma \phi \rangle_{ Y_\delta}.
$$
Since $T_\Gamma : V\to \rL^2(\Gamma)$ and
$P_\delta:\rL^2(\Gamma)\to Y_\delta$ are bounded, the feedback term
$$
(z,\phi)\mapsto \mu \langle P_\delta T_\Gamma z, P_\delta T_\Gamma \phi \rangle_{Y_\delta}
$$
is bounded, symmetric and non-negative on $V\times V$. From \eqref{eq: korn inequality}, it follows that there exists $c_0 > 0$ such that
$$
a_\alpha( z, z ) +  \| z \|_{ \rL^2( \Omega) }^2 \geq c_0 \| z \|_V^2, \qquad z \in V.
$$
Since the feedback term is non-negative, the same estimate holds with
$a_{\mu,\delta}$ in place of $a_\alpha$, i.e.
\begin{equation}
    a_{\mu,\delta} (z,z) + \| z \|_{\rL^2(\Omega)}^2 \geq c_0 \| z \|_{V}^2
    \label{eq:a_mu_delta korn}
\end{equation}
Moreover,
$$
| \langle F_\delta(t) , \phi\rangle_{V',V}|
\leq C ( \| g \|_{\rL^2(\Omega)} + \mu\| y_\delta ( t ) \|_{ Y_\delta} ) \| \phi \|_V,
$$
so $F_\delta \in \rL^2(0,T;V')$ because
$y_\delta \in \rL^2(0,T ; Y_\delta)$. The existence part then follows by the standard Galerkin method for the
two-dimensional NSEs (see for instance \cite{Lions1969} or \cite[Chapter III]{Temam1979}), adapted
to Navier boundary conditions as in \cite[Theorem 6.1]{Kelliher2006}. It gives a weak solution
$$
v \in \rL^\infty(0,T;H) \cap \rL^2(0,T;V) \cap \rC([0,T]; H), \qquad v'\in \rL^{4/3}(0,T;V').
$$

It remains to prove uniqueness. Let $v_1,v_2$ be two solutions with the same
initial datum and set $z:=v_1-v_2$. Then
$$
\frac{1}{2} \frac{ \d }{ \d t} \| z ( t ) \|_{ \rL^2( \Omega) }^2 +a_{ \mu, \delta}( z ( t ), z ( t ) ) =-b( z (t ), v_1( t) ,z ( t)),
$$
because $b( v_2,z, z)=0$. By Ladyzhenskaya's inequality and Young's inequality
$$
|b(z,v_1,z)| \leq C\| v_1 \|_V \| z \|_{ \rL^2( \Omega) }\| z \|_V \leq \frac{ c_0 }{ 2 } \| z \|_V^2 +C \| v_1 \|_V^2 \| z \|_{ \rL^2( \Omega ) }^2   .
$$
Using the estimate \eqref{eq:a_mu_delta korn}, we obtain
$$
\frac{1}{2} \frac{ \d }{ \d t } \| z( t ) \|_{ \rL^2( \Omega ) }^2 +\frac{c_0}{2} \| z ( t ) \|_V^2
\leq C (1 + \| v_1 ( t) \|_V^2 ) \| z ( t) \|_{ \rL^2( \Omega)}^2 .
$$
In particular,
$$
\frac{ \d }{\d t} \| z(t) \|_{\rL^2( \Omega)}^2
\leq C (1+ \| v_1(t) \|_V^2 ) \| z( t )\|_{\rL^2(\Omega)}^2 .
$$
Since $v_1\in\rL^2(0,T;V)$ and $z(0)=0$, Gronwall's lemma gives $z=0$.
Thus the weak solution is unique.
\end{proof}

\subsection{A Gronwall lemma with averaged damping}
We collect here an auxiliary lemma used in the main arguments. It is a consequence of Gronwall's inequality, which converts
positivity of the damping coefficient in long-time average into uniform
exponential decay. 
\begin{lem}
\label{lem: modified gronwall}
Let $E \colon [0,\infty) \to [0 , \infty )$ be absolutely continuous and suppose
that, for some constants $a>0$, $b>0$, and some
$q\in \rL^1_{\mathrm{loc}}(0,\infty)$, one has
\begin{equation}
\frac{\d}{\d t}E(t)+ (a-bq(t))E(t) \leq 0
\label{eq:abstract_averaged_damping_ineq}
\end{equation}
in the sense of distributions on $(0,\infty)$. If
\begin{equation}
b \limsup_{t\to\infty} \frac{1}{t} \int_0^t q(s) \d s < a,
\label{eq: abstract averaged damping hp}
\end{equation}
then there exist constants $C, \gamma >0 $ such that
\begin{equation}
E(t) \leq C e^{-\gamma t} E(0)
\qquad \text{for all } t \geq0.
\label{eq:abstract_averaged_damping_conclusion}
\end{equation}
\end{lem}
\begin{proof}
Choose $\varepsilon > 0 $ so small that
$$
\gamma_0:= a-b\left( \limsup_{t \to \infty} \frac{1}{t} \int_0^t q(s) \d s + \varepsilon \right)>0 .
$$
By the definition of the limsup, there exists $t_\varepsilon>0$ such that
$$
\int_0^t q(s) \d s
\leq \left( \limsup_{\tau \to \infty} \frac{1}{\tau} \int_0^\tau q(s) \d s +\varepsilon \right) t
\qquad \text{for all } t \geq t_\varepsilon .
$$
Integrating \eqref{eq:abstract_averaged_damping_ineq}, we obtain
$$
E(t) \leq E(0) \exp \left(
-a t + b\int_0^t q(s) \d s
\right).
$$
Hence $E(t) \leq E(0)e^{-\gamma_0 t}$ for $t \geq t_\varepsilon $.
Enlarging the multiplicative constant to cover the finite interval
$[0,t_\varepsilon]$ gives \eqref{eq:abstract_averaged_damping_conclusion}.
\end{proof}

{\bf Acknowledgements. }{\small Gianmarco Del Sarto acknowledges the support from the DFG project FOR~5528. Buddhika Priyasad's research was partially supported by the Young Scholar Fund offered by University of Konstanz under the project number: FP 503/26.}


\end{document}